\documentclass[a4paper, 10pt]{article}
\usepackage{amsmath} 
\usepackage{amsthm}
\usepackage{amssymb}
\usepackage{graphics}
\usepackage{latexsym}
\usepackage{comment}
\numberwithin{equation}{section}
\newtheorem{thm}{Theorem}[section]
\newtheorem{prop}[thm]{Proposition}
\newtheorem{lemm}[thm]{Lemma}
\newtheorem{cor}[thm]{Corollary}
\newtheorem{defn}[thm]{Definition}
\newtheorem{rem}[thm]{Remark}
\newcommand{\BBB}{\mathbb}
\newcommand{\R}{{\BBB R}}
\newcommand{\Z}{{\BBB Z}}
\newcommand{\T}{{\BBB T}}

\newcommand{\N}{{\BBB N}}
\newcommand{\C}{{\BBB C}}

\newcommand{\ee}{\mbox{\boldmath $1$}}
\newcommand{\supp}{\operatorname{supp}}

\newcommand{\F}{\mathcal{F}}
\newcommand{\kuuhaku}{\text{}}
\title{Local and global well-posedness for the 2D Zakharov-Kuznetsov-Burgers equation 
in low regularity Sobolev space
}
\author{Hiroyuki Hirayama\\ {\small Organization for Promotion of Tenure Track, University of Miyazaki,} \\ 
{\small Miyazaki, 889-2192, Japan}\\
{\small {\it E-mail address,} h.hirayama@cc.miyazaki-u.ac.jp}
}
\date{}
\begin{document}
\maketitle
\begin{abstract}
In the present paper, we consider the Cauchy problem of the 
2D Zakharov-Kuznetsov-Burgers (ZKB) equation, 
which has the dissipative term $-\partial_x^2u$. 
This is known that the 2D Zakharov-Kuznetsov equation 
is well-posed in $H^s(\R^2)$ for $s>1/2$, 
and the 2D nonlinear parabolic equation with quadratic derivative nonlinearity 
is well-posed in $H^s(\R^2)$ for $s\ge 0$. 
By using the Fourier restriction norm with dissipative effect, 
we prove the well-posedness for ZKB equation in $H^s(\R^2)$ for $s>-1/2$. 
\end{abstract}
%
%
\section{Introduction\label{intro}}
We consider the Cauchy problem of the 2D Zakharov-Kuznetsov-Burgers (ZKB) equation:
\begin{equation}\label{ZKB}
\begin{cases}
\displaystyle \partial_{t}u+\partial_x(\partial_x^2+\partial_y^2)u-\partial_x^2u=\partial_x(u^2),\ \ t>0,\ (x,y)\in \R^2,\\
u(0,x,y)=u_{0}(x,y),\ \ (x,y)\in \R^{2},
\end{cases}
\end{equation}
where the unknown function $u$ is $\R$-valued. 
This equation is two dimensional model of the Kowteweg-de Vries-Burgers (KdVB) equation 
\begin{equation}\label{KdVB}
\displaystyle \partial_{t}u+\partial_x^3u-\partial_x^2u=\partial_x(u^2),\ \ t>0,\ x\in \R, 
\end{equation}
and appears in the dust-ion-acoustic-waves in dusty-plasmas (See, \cite{MS08}, \cite{ZTZSL13}). 
We can see that (\ref{ZKB}) has both dissipative term and dispersive term. 
The aim of this paper is to prove the well-posedness of (\ref{ZKB}) 
in the Sobolev space $H^s(\R^2)$. 

First, we introduce some known results for related problems for 1D case. 
In \cite{KPV96}, Kenig, Ponce, and Vega proved that the Kowteweg-de Vries (KdV) equation
\[
\partial_{t}u+\partial_x^3u=\partial_x(u^2),\ \ t>0,\ x\in \R,
\]
is locally well-posed in $H^s(\R )$ for $s>-3/4$. Colliander, Keel, Stafillani, Takaoka, and Tao (\cite{CKSTT03}) extended the local result to globally in time. 
For the critical case, Kishimoto (\cite{Ki09}) and Guo (\cite{Guo09}) obtained the global well-posedness of KdV equation in $H^{-\frac{3}{4}}(\R)$. 
While, it is proved that the flow map of KdV equation is not uniformly continuous for $s<-3/4$
by Kenig, Ponce,  and Vega in \cite{KPV01} (for $\C$-valued KdV) and Christ, Colliander, and Tao in \cite{CCT03} (for $\R$-valued KdV). 
Therefore, $s=-3/4$ is optimal regularity to obtain the well-posedness of KdV equation by using the iteration argument. 
For the Burgers equation
\[
\partial_{t}u-\partial_x^2u=\partial_x(u^2),\ \ t>0,\ x\in \R,
\]
Dix (\cite{Di96}) proved the local well-posedness in $H^s(\R)$ for $s>-1/2$ and 
nonuniqueness of solution for $s<-1/2$. 
For the critical case, Bekiranov (\cite{Be96}) obtained the local well-posedness of the Burgers equation in $H^{-\frac{1}{2}}(\R )$. 
These results say that $-1/2$ is optimal regularity to obtain the well-posedness of the Burgers equation. 
In \cite{MR02}, Molinet and Ribaud considered the KdV-Burgers equation
\[
\partial_{t}u+\partial_x^3u-\partial_x^2u=\partial_x(u^2),\ \ t>0,\ x\in \R
\]
and obtained the global well-posedness in $H^s(\R)$ for $s>-1$. 
For the critical case, Molinet and Vento (\cite{MV}) proved the global well-posedness of the KdV-Burgers equation in $H^{-1}(\R)$. 
They also proved that the flow map is discontinuous for $s<-1$. 
We note that the regularity $s=-1$ is lower than both $-3/4$ and $-1/2$. 
It means that both the dispersive term and the dissipative term are 
essentially effective for well-posedness.

Next, we introduce some known results for related problems for 2D case. 
Gr\"unrock and Herr (\cite{GH14}), and Molinet and Pilod (\cite{MP15}) proved that the 2D Zakharov-Kuznetsov equation
\begin{equation}\label{ZK}
\partial_{t}u+\partial_x(\partial_x^2+\partial_y^2)u=\partial_x(u^2),\ \ t>0,\ (x,y)\in \R^2
\end{equation}
is locally well-posed in $H^s(\R^2)$ for $s>1/2$. 
Especially, Gr\"unrock and Herr used 
the  linear transform
\[
v(t,x,y)=u\left(t,\frac{4^{\frac{1}{3}}}{2}(x+y),\frac{4^{\frac{1}{3}}}{2\sqrt{3}}(x-y)\right). 
\]
and rewrote (\ref{ZK}) to the symmetric form
\begin{equation}\label{ZK_sym}
\partial_{t}v+(\partial_x^3+\partial_y^3)v=4^{-\frac{1}{3}}(\partial_x+\partial_y)(v^2),\ \ t>0,\ (x,y)\in \R^2.
\end{equation}
Such transform is introduced by Artzi, Koch, and Saut in \cite{AKS03}.  
We note that the well-posedness of (\ref{ZK}) in $H^s(\R^2)$ is 
equivalent to the well-posedness of (\ref{ZK_sym}) in $H^s(\R^2)$. 
This transform is not essentially needed to obtain the well-posedness 
(Actually, Molinet and Pilod did not used such transform), 
but the symmetry helps us to find the structure of the equation
and to write some parts of proof simply. 
Well-posedness of (\ref{ZK}) for $s\le 1/2$ is still open. 
But, Kinoshita gave the author the comment that 
there is a counter example for the $C^2$-well-posedness of (\ref{ZK_sym}) 
in $H^s(\R^2)$ for $s<-1/4$. 
His counter example is given as
\[
\widehat{u_0}(\xi, \eta):=N^{-s+\frac{5}{4}}(\chi_{A}(\xi,\eta)+\chi_{B}(\xi,\eta)), 
\]
where
\[
\begin{split}
A&:=
\left\{\left.Na+N^{-\frac{1}{2}}\delta v+N^{-2}\epsilon \frac{v^{\perp}}{|v^{\perp}|}
\right| -1<\delta, \epsilon<1\right\},\\
B&:=
\left\{\left.Nb+N^{-\frac{1}{2}}\delta v+N^{-2}\epsilon \frac{v^{\perp}}{|v^{\perp}|}
\right| -1<\delta, \epsilon<1\right\},\\
v&:=(3\sqrt[3]{9}, \sqrt[3]{100}),\ 
a:=(\sqrt[3]{2}, \sqrt[3]{75}),\ 
b:=\left(-3\sqrt[3]{2}, -\frac{\sqrt[3]{75}}{5}\right). 
\end{split}
\]
Indeed, we can obtain $\|u_0\|_{H^s}\sim 1$ and
\[
\sup_{0<t\le T}\left\|\int_0^t e^{-(t-t')(\partial_x^3+\partial_y^3)}
(\partial_x+\partial_y)\left((e^{-t'(\partial_x^3+\partial_y^3)}u_0)^2\right)dt'\right\|_{H^s}
\gtrsim N^{-s-\frac{1}{4}}.
\]
While for the nonlinear parabolic equation
\[
\partial_{t}u-\Delta u=P(D)F(u),\ \ t>0,\ (x,y)\in \R^d, 
\]
Ribaud (\cite{R98}) obtained some well-posedness results. 
His results contain that the well-posedness of the 2D nonlinear parabolic equation
\begin{equation}\label{parab}
\partial_{t}u-(\partial_x^2+\partial_y^2) u=\partial_x(u^2),\ \ t>0,\ (x,y)\in \R^2
\end{equation}
in $H^s(\R^2)$ for $s\ge 0$ and nonuniqueness of solution for $s<0$. 
Therefore, our interest is the well-posedness of (\ref{ZKB}) in $H^s(\R^2)$ 
for lower $s$ than both $-1/4$ and $0$. 

Here, we introduce the results for 2D dispersive-dissipative models. 
The KP-Burgers equation
\[
\partial_x\left(\partial_t u+\partial_x^3u-\partial_x^2u-\partial_x (u^2)\right)
+\epsilon \partial_y^2 u=0,\ \ t>0,\ (x,y)\in \R^2,\ \ \epsilon \in \{-1,1\}, 
\]
is also two dimensional model of KdV-Burgers equation. 
We call KP-Burgers equation ``KP-I-Burgers equation'' 
if $\epsilon =-1$, and
``KP-II-Burgers equation''  if $\epsilon =1$. 
The well-posedness of KP-Burgers equation is obtained in $H^{s,0}(\R^2)$ for $s>-1/2$ 
by Kojok in \cite{Koj07} (for $\epsilon =1$) and Mohamad in \cite{Moh12} (for $\epsilon =-1$). 
Where $H^{s,0}(\R^2)$ is anisotropic Sobolev space 
defined by the norm $\|f\|_{H^{s,0}}=\|\langle \xi \rangle^s \widehat{f}(\xi,\eta)\|_{L^2_{\xi\eta}}$. 
Carvajal, Esfahani, and Panthee (\cite{CEP17}) considered the two dimensional 
dissipative KdV type equation
\[
\partial_tu+\partial_x^3u+L_{x,y}u+\partial_x(u^2)=0,\ \ 
t>0,\ (x,y)\in \R^{2}, 
\]
where the operator $L_{x,y}$ is defined by
\[
\F_{xy}[L_{x,y}f](\xi,\eta)=-\Phi (\xi, \eta)\widehat{f}(\xi,\eta)
\]
and the leading term of $\Phi(\xi, \eta)$ is 
$-(|\xi|^{p_1}+|\eta|^{p_2})$ with $p_1$, $p_2>0$. 
They obtained the well-posedness of this equation with $p_2>1$
in $H^{s,0}(\R^2)$ for $s>-3/4$. 
They also considered the high dimensional cases and obtained 
more general results. 
There is no results for the well-posedness of (\ref{ZKB}) as far as we know. 
But the initial-boundary problem of  ZKB equation is studied by Larkin 
(\cite{Lar_arx}, \cite{Lar15}). 

Now, we give the main results in this paper. 
To begin with, we rewrite (\ref{ZKB}) to the symmetric form 
based on \cite{GH14}. 
We put
\[
v(t,x,y)=4u(16t,2(x+y),2\sqrt{3}^{-1}(x-y)). 
\]
Then, (\ref{ZKB}) can be rewritten
\begin{equation}\label{ZKB_sym}
\begin{cases}
\displaystyle \partial_{t}v+(\partial_x^3+\partial_y^3)v-(\partial_x+\partial_y)^2v=(\partial_x+\partial_y)(v^2),\\
v(0,x,y)=v_0(x,y):=4u_{0}(2(x+y),2\sqrt{3}^{-1}(x-y)). 
\end{cases}
\end{equation}
We note that the well-posedness of (\ref{ZKB}) in $H^s(\R^2)$ is 
equivalent to the well-posedness of (\ref{ZKB_sym}) in $H^s(\R^2)$. 
Therefore, we consider (\ref{ZKB_sym}) instead of (\ref{ZKB}). 
\begin{thm}\label{LWP}
\ Let $s>-\frac{1}{2}$. Then (\ref{ZKB_sym}) is locally well-posed in $H^s(\R^2)$. 
(Therefore (\ref{ZKB}) is also locally well-posed in $H^s(\R^2)$.) 
More precisely, for any $v_0\in H^s(\R^2)$, there exist $T>0$, 
and an unique solution $v\in X^{s,\frac{1}{2},1}_T\ (\hookrightarrow C([0,T];H^s(\R^2))$\ 
$($See, Definition~\ref{FRN}$)$
 to (\ref{ZKB_sym}) in $[0,T]$. 
Furthermore, 
the data-to-solution map is Lipschitz continuous from 
$H^s(\R^2)$ to $C([0,T];H^s(\R^2))$. 
\end{thm}
\begin{thm}\label{GWP}
Let $s>-\frac{1}{2}$. For any $v_0\in \widetilde{H}^{s}(\R^2)$, the solution $v$ obtained in Theorem~\ref{LWP}
can be extended globally in time and $v$ belongs to $C((0,\infty );\widetilde{H}^{\infty}(\R^2))$, 
where $\widetilde{H}^s(\R^2)$ is the completion of the Schwartz class 
$\mathcal{S} (\R^2)$ with the norm $\|f\|_{\widetilde{H}^s}=\|\langle \xi +\eta \rangle^s \widehat{f}(\xi,\eta)\|_{L^2_{\xi\eta}}$, 
and $\widetilde{H}^{\infty}(\R^2)=\bigcap_{s\in \R}\widetilde{H}^{s}(\R^2)$. 
\end{thm}
\begin{rem}
{\rm (i)}\ Although (\ref{ZKB}) does not have 
the dissipative term with respect to $y$,  
the well-posedness of (\ref{ZKB}) is obtained in isotropic 
Sobolev space $H^{s}(\R^2)$
for lower regularity than both (\ref{ZK}) and (\ref{parab}). \\
{\rm (ii)}\ Theorem~\ref{GWP} says that (\ref{ZKB}) is 
globally well-posed in $H^{s,0}(\R^2)$ for $s>-\frac{1}{2}$. 
\end{rem}

To obtain Theorem~\ref{LWP}, we have to treat the dissipative term carefully, 
because the symbol $(\xi+\eta)^2$ is vanished on the line $\{(\xi, -\xi)|\xi\in \R\}$. 
But the nonlinear term is also vanished on the same line. 
It helps us to obtain the key bilinear estimate (Proposition~\ref{bilin_est}). 
We will use the iteration argument with the Fourier restriction norm 
to obtain the local well-posedness. 
While, the global well-posedness will be proved by using 
the smoothing effect from the dissipative term and 
non-increasing of $L^2$-norm of the solution. 

\kuuhaku \\
\noindent {\bf Notation.} 
We denote the spatial Fourier transform by\ \ $\widehat{\cdot}$\ \ or $\F_{xy}$, 
the Fourier transform in time by $\F_{t}$, and the Fourier transform in all variables by\ \ $\widetilde{\cdot}$\ \ or $\F$. 
The operator $U(t)=e^{-t(\partial_x^3+\partial_y^3)}$ and $W(t)=e^{|t|(\partial_x+\partial_y)^2}e^{-t(\partial_x^3+\partial_y^3)}$ on $H^{s}(\R^2)$
is given as a Fourier multiplier
\[
\F_{xy}[U(t)f](\xi,\eta)=e^{it(\xi^3+\eta^3)}\widehat{f}(\xi ),\ \ \F_{xy}[W(t)f](\xi,\eta)=e^{-|t|(\xi+\eta)^2}e^{it(\xi^3+\eta^3)}\widehat{f}(\xi ). 
\]
$U(t)$ and $W(t)$ give a solution to
\[
\partial_t u+(\partial_x^3+\partial_y^3)u=0
\]
and
\[
\partial_t u+(\partial_x^3+\partial_y^3)u-{\rm sgn}(t)(\partial_x+\partial_y)^2u=0
\]
respectively. 
We note that $\F[U(-\cdot )F(\cdot )](\tau ,\xi ,\eta )=\widetilde{F}(\tau +\xi^3+\eta^3,\xi,\eta )$. 

We will use $A\lesssim B$ to denote an estimate of the form $A \le CB$ for some constant $C$ and write $A \sim B$ to mean $A \lesssim B$ and $B \lesssim A$. 
We will use the convention that capital letters denote dyadic numbers, e.g. $N=2^{n}$ for $n\in \Z$ and for a dyadic summation we write
$\sum_{N}a_{N}:=\sum_{n\in \Z}a_{2^{n}}$, $\sum_{N\geq N'}a_{N}:=\sum_{n\in \Z, 2^{n}\geq N'}a_{2^{n}}$, and  $\sum_{N\leq N'}a_{N}:=\sum_{n\in \Z, 2^{n}\leq N'}a_{2^{n}}$ for brevity. 
Let $\chi \in C^{\infty}_{0}((-2,2))$ be an even, non-negative function such that $\chi (t)=1$ for $|t|\leq 1$. 
We define $\varphi (t):=\chi (t)-\chi (2t)$ and $\varphi_{N}(t):=\varphi (N^{-1}t)$. Then, $\sum_{N}\varphi_{N}(t)=1$ whenever $t\neq 0$. 
We define the projections
\[
\begin{split}
&\widehat{P_{N}u}(\xi ,\eta):=\varphi_{N}(|(\xi,\eta )| )\widehat{u}(\xi,\eta),\ \widehat{P_{N,M}u}(\xi ,\eta):=\varphi_{N,M}(\xi ,\eta )\widehat{u}(\xi ,\eta),\\
&\widetilde{Q_{L}u}(\tau ,\xi ,\eta):=\varphi_{L}(\tau -\xi^3-\eta^3)\widetilde{u}(\tau ,\xi ,\eta), 
\end{split}
\]
where $\varphi_{N,M}(\xi ,\eta):=\varphi_{N}(|(\xi ,\eta )|)\varphi_M(\xi+\eta )$. 

The rest of this paper is planned as follows.
In Section 2, we will give the definition of the solution space, and prove the linear estimates. 
In Section 3, we will prove the bilinear estimate which is main part of this paper. 
In Section 4, we will give the proof of the well-posedness (Theorems~\ref{LWP} and ~\ref{GWP}). 

%
%
\section{Function space and linear estimate}
In this section, we define the function space, 
and prove the estimate for linear solution and Duhamel term. 
First, we consider the standard Fourier restriction norm $\|\cdot \|_{X^{s,b}}$ for (\ref{ZKB_sym}) defined by
\[
\|u\|_{X^{s,b}}=\|\langle |(\xi ,\eta )|\rangle^s\langle (\xi +\eta )^2+i(\tau -\xi^3-\eta^3)\rangle^b\widetilde{u}(\tau ,\xi ,\eta )\|_{L^2_{\tau \xi \eta}}.
\]
Such Fourier restriction norm was introduced by J. Bourgain (\cite{Bo93}) 
for the nonlinear Schr\"odinger equation and the KdV equation. 
Let $\psi \in C^{\infty}(\R )$ denotes a cut-off function such that $\supp \psi \subset [-2,2]$, $\psi =1$ on $[-1,1]$. 
We note that, the estimate
\[
\|\psi (t)W(t)u_0\|_{X^{s,b}}\lesssim \|\langle |(\xi ,\eta )|\rangle^{s}\langle \xi +\eta \rangle^{b-\frac{1}{2}} \widehat{u_0}(\xi ,\eta  )\|_{L^2_{\xi \eta}}
\]
holds. Therefore, if $b\le 1/2$, then $\psi W(\cdot )u_0\in X^{s,b}$ for $u_0\in H^s$. 
But the embedding $X^{s,b}\hookrightarrow C(\R;H^s(\R^2))$ 
does not hold for $b\le 1/2$. 
Therefore, we use the Besov type Fourier restriction norm defined as follows.
\begin{defn}\label{FRN}
Let $s\in \R$, $b\in \R$. \\
(i)\ We define the function space $X^{s,b,1}$ as the completion of the Schwartz class ${\mathcal S}(\R_{t}\times \R^2_{x,y})$ with the norm
\[
\|u\|_{X^{s,b,1}}=\left\{\sum_{N\in 2^{\Z}}\sum_{M\in 2^{\Z}}\left(\sum_{L\in 2^{\Z}}\langle N\rangle^s\langle M^2+L\rangle^{b}\|P_{N,M}Q_{L}u\|_{L^2_{txy}}\right)^2\right\}^{\frac{1}{2}}.
\]
(ii)\ For $T>0$, we define the time localized space $X^{s,b,1}_T$ as
\[
X^{s,b,1}_{T}=\{u|_{[0,T]}|u\in X^{s,b,1}\}
\]
with the norm
\[
\|u\|_{X^{s,b,1}_T}=\inf \{\|v\|_{X^{s,b,1}}|v\in X^{s,b,1},\ v|_{[0,T]}=u|_{[0,T]}\}.
\]
\end{defn}
\begin{rem}
(i)\ The embedding $X^{s,\frac{1}{2},1}_T\hookrightarrow C([0,T];H^s(\R^2))$ holds. \\
(ii)\ The size of $|\xi +\eta |$, which comes from the symbol of the dissipative term 
of (\ref{ZKB_sym}), 
 is not decided by the size of $|(\xi ,\eta )|$. 
Therefore, to use the dissipative effect strictly,  
we focus on not only $|(\xi ,\eta )|\sim N$, but also $|\xi +\eta |\sim M$. 
This is a different point from 1D case. \\
(iii)\ We can assume $\sum_{M\in 2^{\Z}}=\sum_{M\lesssim N}$ since $|\xi +\eta|\lesssim |(\xi ,\eta )|$ holds.
\end{rem}
We choose $X^{s,\frac{1}{2},1}_T$ as the solution space. 
Now, we define the operator ${\mathcal K}$ and ${\mathcal L}$ by
\[
\begin{split}
&{\mathcal K}F(t)(\xi,\eta ):=\int_{\R}\frac{e^{it\tau}-e^{-|t|(\xi +\eta )^2}}{(\xi +\eta)^2+i\tau}\F[U(-\cdot )F(\cdot )](\tau ,\xi ,\eta )d\tau\\
&{\mathcal L}F(t):=U(t)\int_{\R^2}e^{ix\xi}e^{iy\eta}{\mathcal K}F(t)(\xi,\eta )d\xi d\eta =U(t)\F_{x,y}^{-1}[{\mathcal K}F(t)].
\end{split}
\]
Then, we note that
\[
{\mathcal L}F(t)=\int_0^t W(t-t')F(t')dt'
\]
holds for $t\ge 0$ and the integral form of (\ref{ZKB_sym}) on $[0,\infty)$ is given by
\begin{equation}\label{ZKB_sym_int}
\begin{split}
v(t)&=W(t)v_0+\int_0^tW(t-t')(\partial_x+\partial_y)(v(t')^2)dt'\\
&=W(t)v_0+{\mathcal L}((\partial_x+\partial_y)v)(t).
\end{split}
\end{equation}
\begin{prop}\label{lin_est}
Let $s\in \R$. There exists $C_1>0$, such that for any $u_0\in H^s(\R^2)$, we have
\[
\|\psi (t)W(t)u_0\|_{X^{s,\frac{1}{2},1}}\le C_1\|u_0\|_{H^s}. 
\]
\end{prop}
\begin{proof}
Since
\[
\left(\sum_N\sum_M\langle N\rangle^{2s}\|P_{N,M}u_0\|_{L^2_{xy}}^2\right)^{\frac{1}{2}}\sim \|u_0\|_{H^s}
\]
holds, it suffice to prove
\[
\sum_{L}\langle M^2+L\rangle^{\frac{1}{2}}\|P_{N,M}Q_{L}(\psi (t)W(t)u_0)\|_{L^2_{txy}}\lesssim \|P_{N,M}u_0\|_{L^2_{xy}}
\]
for each $N$, $M\in 2^{\Z}$. By using Plancherel's theorem, we have
\[
\begin{split}
&\|P_{N,M}Q_{L}(\psi (t)W(t)u_0)\|_{L^2_{txy}}\\
&\sim \|\varphi_{N,M}(\xi, \eta )\varphi_L(\tau )\F_t[\psi (t)e^{-|t|(\xi +\eta )^2}]\widehat{u_0}(\xi ,\eta )\|_{L^2_{\xi \eta t}}\\
&\lesssim \|P_{N,M}u_0\|_{L^2_{xy}}
\|\phi_M(\xi +\eta )\varphi_L(\tau )\F_t[\psi (t)e^{-|t|(\xi +\eta )^2}]\|_{L^{\infty}_{\xi \eta}L^2_t}\\
&=\|P_{N,M}u_0\|_{L^2_{xy}}
\|\phi_M(\zeta )\varphi_L(\tau )\F_t[\psi (t)e^{-|t|\zeta^2}]\|_{L^{\infty}_{\zeta}L^2_t},
\end{split}
\]
where $\phi_M=\varphi_{2M}+\varphi_M+\varphi_{\frac{M}{2}}$ and we used $\varphi_M=\varphi_M\phi_M$. Therefore, it suffice to prove
\begin{equation}\label{exp_besov}
\sum_{L}\langle M^2+L\rangle^{\frac{1}{2}}\|\phi_M(\zeta)\varphi_L(\tau )\F_t[\psi (t)e^{-|t|\zeta^2}]\|_{L^{\infty}_{\zeta}L^2_{\tau}}
\lesssim 1.
\end{equation}
It is obtained in the proof of Proposition\ 4.1 in \cite{MV}. 
\end{proof}
\begin{prop}\label{duam_est}
Let $s\in \R$. There exists $C_2>0$, such that for any $F\in X^{s,-\frac{1}{2},1}$, we have
\[
\left\|\psi (t){\mathcal L}F(t)\right\|_{X^{s,\frac{1}{2},1}}\le C_2\|F\|_{X^{s,-\frac{1}{2},1}}
\]
\end{prop}
\begin{proof}
We use the argument in the proof of Lemma 4.1 in \cite{MV}. 
Since
\[
\|P_{N,M}Q_L(\psi (t){\mathcal L}F(t))\|_{L^2_{txy}}
\sim \|\varphi_{N,M}(\xi,\eta)\varphi_L(\tau )\F_t[\psi {\mathcal K}F](\tau,\xi,\eta)\|_{L^2_{\xi\eta\tau}},
\]
it suffice to show that
\begin{equation}\label{Duamel__est_pf}
\begin{split}
&\sum_L\langle M^2+L\rangle^{\frac{1}{2}}\|\varphi_{N,M}(\xi,\eta)\varphi_L(\tau )\F_t[\psi {\mathcal K}F](\tau,\xi,\eta)\|_{L^2_{\xi\eta\tau}}\\
&\lesssim \sum_L\langle M^2+L\rangle^{-\frac{1}{2}}\|\varphi_{N,M}(\xi,\eta)\varphi_L(\tau )\F[U(-\cdot )F(\cdot )](\tau,\xi,\eta)\|_{L^2_{\xi\eta\tau}}
\end{split}
\end{equation}
We put $w(t)=U(-t)F(t)$ and split $\psi{\mathcal K}F$ into $K_{1}+K_{2}+K_{3}-K_{4}$, where
\[
\begin{split}
K_{1}(t,\xi,\eta)&=\psi(t)\int_{|\tau |\le 1}\frac{e^{it\tau}-1}{(\xi +\eta)^2+i\tau}\widetilde{w}(\tau ,\xi ,\eta )d\tau,\\
K_{2}(t,\xi,\eta)&=\psi(t)\int_{|\tau |\le 1}\frac{1-e^{-|t|(\xi+\eta)^2}}{(\xi +\eta)^2+i\tau}\widetilde{w}(\tau ,\xi ,\eta )d\tau,\\
K_{3}(t,\xi,\eta)&=\psi(t)\int_{|\tau |\ge 1}\frac{e^{it\tau}}{(\xi +\eta)^2+i\tau}\widetilde{w}(\tau ,\xi ,\eta )d\tau,\\
K_{4}(t,\xi,\eta)&=\psi(t)\int_{|\tau |\ge 1}\frac{e^{-|t|(\xi+\eta)^2}}{(\xi +\eta)^2+i\tau}\widetilde{w}(\tau ,\xi ,\eta )d\tau.
\end{split}
\]
Furthermore, we put $w_{N,M}=P_{N,M}w$. 
We note that $\widetilde{w}_{N,M}(\tau,\xi,\eta)=\phi_M(\xi+\eta)\widetilde{w}_{N,M}(\tau,\xi,\eta)$ since $\varphi_M=\varphi_M\phi_M$.\\
\kuuhaku \\
\underline{Estimate for $K_1$}

By using the Taylor expansion, we have
\[
\begin{split}
&\|\varphi_{N,M}(\xi,\eta)\varphi_L(\tau )\F_t[K_1](\tau,\xi,\eta)\|_{L^2_{\xi\eta\tau}}\\
&\lesssim \sum_{n=1}^{\infty}\frac{1}{n!}\left\|\left(\int_{|\tau |\le 1}\frac{|\tau |^n|\widetilde{w}_{N,M}(\tau,\xi,\eta)|}{(\xi+\eta)^2+|\tau|}d\tau \right)
\|\varphi_L(\tau)\F_t[t^n\psi (t)](\tau)\|_{L^2_\tau}\right\|_{L^2_{\xi\eta}}.
\end{split}
\]
By the Cauchy-Schwartz inequality, we obtain
\[
\begin{split}
&\int_{|\tau |\le 1}\frac{|\tau |^n|\widetilde{w}_{N,M}(\tau,\xi,\eta)|}{(\xi+\eta)^2+|\tau|}d\tau \\
&\lesssim 
\left(\int_{|\tau |\le 1}\frac{|\tau |^2\langle(\xi+\eta)^2+|\tau |\rangle}{((\xi+\eta)^2+|\tau |)^2}|\phi_M(\xi +\eta)|^2d\tau\right)^{\frac{1}{2}}
\left(\int_{|\tau |\le 1}\frac{|\widetilde{w}_{N,M}(\tau ,\xi ,\eta )|^2}{ \langle(\xi+\eta)^2+|\tau |\rangle}d\tau\right)^{\frac{1}{2}}\\
&\lesssim \langle M\rangle^{-1} \sum_{L}\langle M^2+L\rangle^{-\frac{1}{2}}\|\varphi_L(\tau )\widetilde{w}_{N,M}(\tau, \xi, \eta)\|_{L^2_{\tau}}
\end{split}
\]
for $n\ge 1$. Therefore, we get 
\[
\begin{split}
&\sum_L\langle M^2+L\rangle^{\frac{1}{2}}\|\varphi_{N,M}(\xi,\eta)\varphi_L(\tau )\F_t[K_1](\tau,\xi,\eta)\|_{L^2_{\xi\eta\tau}}\\
&\lesssim \sum_{n=1}^{\infty}\frac{1}{n!}
\||t|^n\psi \|_{B^{\frac{1}{2}}_{2,1}}
\sum_L\langle M^2+L\rangle^{-\frac{1}{2}}\|\varphi_L(\tau )\widetilde{w}_{N,M}(\tau,\xi,\eta)\|_{L^2_{\xi\eta\tau}}\\
&\lesssim \sum_L\langle M^2+L\rangle^{-\frac{1}{2}}\|\varphi_L(\tau )\widetilde{w}_{N,M}(\tau,\xi,\eta)\|_{L^2_{\xi\eta\tau}}
\end{split}
\]
since $\langle M^2+L\rangle^{\frac{1}{2}}\langle M\rangle^{-1}\lesssim \langle L\rangle^{\frac{1}{2}}$. \\
\kuuhaku \\
\underline{Estimate for $K_2$}

By Plancherel's theorem, we have
\[
\begin{split}
&\|\varphi_{N,M}(\xi,\eta)\varphi_L(\tau )\F_t[K_2](\tau,\xi,\eta)\|_{L^2_{\xi\eta\tau}}\\
&\lesssim \left\|\left(\int_{|\tau |\le 1}\frac{|\widetilde{w}_{N,M}(\tau,\xi,\eta)|}{(\xi+\eta)^2+|\tau|}d\tau \right)
\|\phi_M(\xi+\eta)\varphi_L(\tau)\F_t[\psi (t)(1-e^{-|t|(\xi+\eta)^2})](\tau)\|_{L^2_\tau}\right\|_{L^2_{\xi\eta}}.
\end{split}
\]
By the Cauchy-Schwartz inequality, we obtain
\[
\begin{split}
&\int_{|\tau |\le 1}\frac{|\widetilde{w}_{N,M}(\tau,\xi,\eta)|}{(\xi+\eta)^2+|\tau|}d\tau \\
&\lesssim 
\left(\int_{|\tau |\le 1}\frac{\langle(\xi+\eta)^2+|\tau |\rangle}{((\xi+\eta)^2+|\tau |)^2}|\phi_M(\xi +\eta)|^2d\tau\right)^{\frac{1}{2}}
\left(\int_{|\tau |\le 1}\frac{|\widetilde{w}_{N,M}(\tau ,\xi ,\eta )|^2}{ \langle(\xi+\eta)^2+|\tau |\rangle}d\tau\right)^{\frac{1}{2}}\\
&\lesssim M^{-2}\langle M\rangle \sum_{L}\langle M^2+L\rangle^{-\frac{1}{2}}\|\varphi_L(\tau )\widetilde{w}_{N,M}(\tau, \xi, \eta)\|_{L^2_{\tau}}
\end{split}
\]
Therefore if $M\ge 1$, then we get 
\[
\begin{split}
&\sum_L\langle M^2+L\rangle^{\frac{1}{2}}\|\varphi_{N,M}(\xi,\eta)\varphi_L(\tau )\F_t[K_2](\tau,\xi,\eta)\|_{L^2_{\xi\eta\tau}}\\
&\lesssim \sum_L\langle M^2+L\rangle^{-\frac{1}{2}}\|\varphi_L(\tau )\widetilde{w}_{N,M}(\tau,\xi,\eta)\|_{L^2_{\xi\eta\tau}}
\end{split}
\]
by (\ref{exp_besov}) and
\[
\sum_{L}\langle M^2+L\rangle^{\frac{1}{2}}\|\varphi_L(\tau )\F_t[\psi ](\tau)\|_{L^2_{\tau}}
\lesssim M\|\psi\|_{B^{\frac{1}{2}}_{2,1}}\lesssim M.
\]
While if $M\le 1$, then by using the Taylor expansion, we have
\[
\begin{split}
&\|\phi_M(\xi+\eta)\varphi_L(\tau)\F_t[\psi (t)(1-e^{-|t|(\xi+\eta)^2})](\tau)\|_{L^2_\tau}\\
&\lesssim \sum_{n=1}^{\infty}\frac{(\xi +\eta )^{2n}}{n!}
\phi_M(\xi+\eta)\|\varphi_L(\tau )\F_t[\psi (t)|t|^n](\tau )\|_{L^2_{\tau}}\\
&\lesssim M^2\sum_{n=1}^{\infty}\frac{1}{n!}\|\varphi_L(\tau )\F_t[\psi (t)|t|^n](\tau )\|_{L^2_{\tau}}
\end{split}
\]
Therefore, we get
\[
\begin{split}
&\sum_L\langle M^2+L\rangle^{\frac{1}{2}}\|\varphi_{N,M}(\xi,\eta)\varphi_L(\tau )\F_t[K_2](\tau,\xi,\eta)\|_{L^2_{\xi\eta\tau}}\\
&\lesssim \sum_{n=1}^{\infty}\frac{1}{n!}\||t|^n\psi\|_{B^{\frac{1}{2}}_{2,1}}\sum_L\langle M^2+L\rangle^{-\frac{1}{2}}\|\varphi_L(\tau )\widetilde{w}_{N,M}(\tau,\xi,\eta)\|_{L^2_{\xi\eta\tau}}\\
&\lesssim \sum_L\langle M^2+L\rangle^{-\frac{1}{2}}\|\varphi_L(\tau )\widetilde{w}_{N,M}(\tau,\xi,\eta)\|_{L^2_{\xi\eta\tau}}.
\end{split}
\]
\\
\underline{Estimate for $K_3$}

We put $g_{N,M}(t)=\F_t^{-1}[\ee_{|\tau |\ge 1}((\xi +\eta)^2+i\tau)^{-1}\widetilde{w}_{N,M}(\tau ,\xi ,\eta )](t)$. 
Then, we have
\[
\begin{split}
|\varphi_{N,M}(\xi,\eta)\varphi_L(\tau )\F_t[K_3](\tau)|
&\sim |\varphi_L(\tau)\left(\F_t[\psi]*_{\tau}\F_t[g_{N,M}](\tau )\right)|\\
&\lesssim \sum_{L_1}\sum_{L_2}|\varphi_L(\tau )(\varphi_{L_1}\F_t[\psi])*_{\tau}(\varphi_{L_2}\F_t[g_{N,M})(\tau )|
\end{split}
\]
\\
(i)\ Summation for $L_1\ll L$\ (then, $L_2\sim L$.)

By the Young inequality, we have
\[
\begin{split}
&\|\varphi_L(\tau )(\varphi_{L_1}\F_t[\psi])*_{\tau}(\varphi_{L_2}\F_t[g_{N,M})(\tau )\|_{L^2_{\tau}}\\
&\lesssim \|\varphi_{L_1}(\tau )\F_t[\psi](\tau )\|_{L^1_{\tau}}\|\varphi_{L_2}(\tau )\F_t[g_{N,M}](\tau )\|_{L^2_{\tau}}\\
&\lesssim \|\varphi_{L_1}(\tau )\F_t[\psi](\tau )\|_{L^1_{\tau}}
\langle M^2+L_2\rangle^{-1}\|\varphi_{L_2}(\tau )\widetilde{w}_{N,M}(\tau )\|_{L^2_{\tau}}.
\end{split}
\]
Therefore, we obtain
\[
\begin{split}
&\sum_{L}\langle M^2+L\rangle^{\frac{1}{2}}\sum_{L_1\ll L}\sum_{L_2\sim L}
\|\varphi_L(\tau )(\varphi_{L_1}\F_t[\psi])*_{\tau}(\varphi_{L_2}\F_t[g_{N,M}])(\tau )\|_{L^2_{\xi\eta\tau}}\\
&\lesssim \left(\sum_{L_1}\|\varphi_{L_1}(\tau )\F_t[\psi](\tau )\|_{L^1_{\tau}}\right)
\left(\sum_{L_2}\langle M^2+L_2\rangle^{-\frac{1}{2}}\|\varphi_{L_2}(\tau )\widetilde{w}_{N,M}(\tau,\xi,\eta )\|_{L^2_{\xi\eta\tau}}\right)\\
&\lesssim \sum_{L_2}\langle M^2+L_2\rangle^{-\frac{1}{2}}\|\varphi_{L_2}(\tau )\widetilde{w}_{N,M}(\tau,\xi,\eta )\|_{L^2_{\xi\eta\tau}}
\end{split}
\]
since
\[
\sum_{L_1}\|\varphi_{L_1}(\tau )\F_t[\psi](\tau )\|_{L^1_{\tau}}
\lesssim \sum_{L_1}L_1^{\frac{1}{2}}\|\varphi_{L_1}(\tau )\F_t[\psi](\tau )\|_{L^2_{\tau}}\lesssim \|\psi\|_{B^{\frac{1}{2}}_{2,1}}\lesssim 1.
\]
\\
(ii)\ Summation for $L\lesssim M^2$, $L_1\gtrsim L$.

By the H\"older inequality and the Young inequality, we have
\[
\begin{split}
&\|\varphi_L(\tau )(\varphi_{L_1}\F_t[\psi])*_{\tau}(\varphi_{L_2}\F_t[g_{N,M}])(\tau )\|_{L^2_{\tau}}\\
&\lesssim \|\varphi_L\|_{L^2_{\tau}}\|\varphi_{L_1}(\tau )\F_t[\psi](\tau )\|_{L^2_{\tau}}\|\varphi_{L_2}(\tau )\F_t[g_{N,M}](\tau )\|_{L^2_{\tau}}\\
&\lesssim L^{\frac{1}{2}}\|\varphi_{L_1}(\tau )\F_t[\psi](\tau )\|_{L^2_{\tau}}\langle M^2+L_2\rangle^{-1}\|\varphi_{L_2}(\tau )\widetilde{w}_{N,M}(\tau )\|_{L^2_{\tau}}.
\end{split}
\]
Therefore, we obtain
\[
\begin{split}
&\sum_{L\lesssim M^2}\langle M^2+L\rangle^{\frac{1}{2}}\sum_{L_1\gtrsim L}\sum_{L_2}
\|\varphi_L(\tau )(\varphi_{L_1}\F_t[\psi])*_{\tau}(\varphi_{L_2}\F_t[g_{N,M}])(\tau )\|_{L^2_{\xi\eta\tau}}\\
&\lesssim \langle M\rangle \left(\sum_{L_1}L_1^{\frac{1}{2}}\|\varphi_{L_1}(\tau )\F_t[\psi](\tau )\|_{L^2_{\tau}}\right)
\left(\sum_{L_2}\langle M^2+L_2\rangle^{-1}\|\varphi_{L_2}(\tau )\widetilde{w}_{N,M}(\tau,\xi,\eta )\|_{L^2_{\xi\eta\tau}}\right)\\
&\lesssim \sum_{L_2}\langle M^2+L_2\rangle^{-\frac{1}{2}}\|\varphi_{L_2}(\tau )\widetilde{w}_{N,M}(\tau,\xi,\eta )\|_{L^2_{\xi\eta\tau}}
\end{split}
\]
since $\langle M\rangle \lesssim \langle M^2+L_2\rangle^{\frac{1}{2}}$ and 
\[
\sum_{L_1}L_1^{\frac{1}{2}}\|\varphi_{L_1}(\tau )\F_t[\psi](\tau )\|_{L^2_{\tau}}\lesssim \|\psi\|_{B^{\frac{1}{2}}_{2,1}}\lesssim 1.
\]
\\
(iii)\ Summation for $L_1\gtrsim L\gtrsim M^2$.
By the Young inequality and the Cauchy-Schwartz inequality, we have
\[
\begin{split}
&\|\varphi_L(\tau )(\varphi_{L_1}\F_t[\psi])*_{\tau}(\varphi_{L_2}\F_t[g_{N,M}])(\tau )\|_{L^2_{\tau}}\\
&\lesssim \|\varphi_{L_1}(\tau )\F_t[\psi](\tau )\|_{L^2_{\tau}}\|\varphi_{L_2}(\tau )\F_t[g_{N,M}](\tau )\|_{L^1_{\tau}}\\
&\lesssim \|\varphi_{L_1}(\tau )\F_t[\psi](\tau )\|_{L^2_{\tau}}
\langle M^2+L_2\rangle^{-\frac{1}{2}}\|\varphi_{L_2}(\tau )\widetilde{w}_{N,M}(\tau )\|_{L^2_{\tau}}.
\end{split}
\]
Therefore, we obtain
\[
\begin{split}
&\sum_{L\gtrsim M^2}\langle M^2+L\rangle^{\frac{1}{2}}\sum_{L_1\gtrsim L}\sum_{L_2}
\|\varphi_L(\tau )(\varphi_{L_1}\F_t[\psi])*_{\tau}(\varphi_{L_2}\F_t[g_{N,M}])(\tau )\|_{L^2_{\xi\eta\tau}}\\
&\lesssim \left(\sum_{L_1}\langle L_1\rangle^{\frac{1}{2}}\|\varphi_{L_1}(\tau )\F_t[\psi](\tau )\|_{L^2_{\tau}}\right)
\left(\sum_{L_2}\langle M^2+L_2\rangle^{-\frac{1}{2}}\|\varphi_{L_2}(\tau )\widetilde{w}_{N,M}(\tau,\xi,\eta )\|_{L^2_{\xi\eta\tau}}\right)\\
&\lesssim \sum_{L_2}\langle M^2+L_2\rangle^{-\frac{1}{2}}\|\varphi_{L_2}(\tau )\widetilde{w}_{N,M}(\tau,\xi,\eta )\|_{L^2_{\xi\eta\tau}}
\end{split}
\]
since
\[
\sum_{L_1}\langle L_1\rangle^{\frac{1}{2}}\|\varphi_{L_1}(\tau )\F_t[\psi](\tau )\|_{L^2_{\tau}}
\lesssim \|\psi \|_{B^{\frac{1}{2}}_{2,1}}\lesssim 1.
\]
\\
\underline{Estimate for $K_4$}

By Plancherel's theorem, we have
\[
\begin{split}
&\|\varphi_{N,M}(\xi,\eta)\varphi_L(\tau )\F_t[K_4](\tau,\xi,\eta)\|_{L^2_{\xi\eta\tau}}\\
&\lesssim \left\|\left(\int_{|\tau |\ge 1}\frac{|\widetilde{w}_{N,M}(\tau,\xi,\eta)|}{(\xi+\eta)^2+|\tau|}d\tau \right)
\|\phi_M(\xi+\eta)\varphi_L(\tau)\F_t[\psi (t)e^{-|t|(\xi+\eta)^2}](\tau)\|_{L^2_\tau}\right\|_{L^2_{\xi\eta}}.
\end{split}
\]
By the Cauchy-Schwartz inequality, we obtain
\[
\begin{split}
\int_{|\tau |\ge 1}\frac{|\widetilde{w}_{N,M}(\tau,\xi,\eta)|}{(\xi+\eta)^2+|\tau|}d\tau
&\lesssim \sum_L\langle M^2+L\rangle^{-1}\|\varphi_L(\tau )\widetilde{w}_{N,M}(\tau,\xi,\eta)\|_{L^1_{\tau}}\\
&\lesssim \sum_L\langle M^2+L\rangle^{-\frac{1}{2}}\|\varphi_L(\tau )\widetilde{w}_{N,M}(\tau,\xi,\eta)\|_{L^2_{\tau}}.
\end{split}
\]
Therefore, by (\ref{exp_besov}), we get 
\[
\begin{split}
&\sum_L\langle M^2+L\rangle^{\frac{1}{2}}\|\varphi_{N,M}(\xi,\eta)\varphi_L(\tau )\F_t[K_4](\tau,\xi,\eta)\|_{L^2_{\xi\eta\tau}}\\
&\lesssim \sum_L\langle M^2+L\rangle^{-\frac{1}{2}}\|\varphi_L(\tau )\widetilde{w}_{N,M}(\tau,\xi,\eta)\|_{L^2_{\xi\eta\tau}}.
\end{split}
\]
\end{proof}
%
%
\section{Bilinear estimate}
In this section, we prove the estimate for nonlinear term as follows. 
\begin{prop}\label{bilin_est}
Let $s\ge s_0>-\frac{1}{2}$. There exist $0<\delta \ll 1$ and $C_3>0$, such that for any $u$, $v\in X^{s,\frac{1-\delta}{2},1}$, we have
\[
\|(\partial_x+\partial_y)(uv)\|_{X^{s,-\frac{1}{2},1}}
\le C_3\|u\|_{X^{s,\frac{1-\delta}{2},1}}\|v\|_{X^{s,\frac{1-\delta}{2},1}}.
\]
\end{prop}
To prove Proposition~\ref{bilin_est}, we first give some Strichartz estimates.
\begin{prop}\label{Stri}
Let 
$(p,q)\in \R^2$ satisfy $p\ge 3$ and $\frac{3}{p}+\frac{2}{q}=1$. 
For any $u_0\in L^2(\R^2)$, we have
\[
\|U(t)u_0\|_{L^p_tL^q_{xy}}\lesssim \|u_0\|_{L^2_{xy}}.
\]
\end{prop}
Proposition~\ref{Stri} is obtained by using 
the variable transform $(x,y)\mapsto (4^{-\frac{1}{3}}(x+\sqrt{3}y), 4^{-\frac{1}{3}}(x-\sqrt{3}y))$ 
in Proposition\ 2.4 in \cite{LP09}.
\begin{prop}\label{mod_Stri}
For any $u_0\in L^2(\R^2)$, we have
\[
\|D_x^{\frac{1}{8}}D_y^{\frac{1}{8}}U(t)u_0\|_{L^4_{txy}}\lesssim \|u_0\|_{L^2_{xy}}, 
\]
where $D_x^s=\F_{xy}^{-1}|\xi |^s\F_{xy}$, $D_y^s=\F_{xy}^{-1}|\eta |^s\F_{xy}$ for $s\in \R$. 
\end{prop}
Proposition~\ref{mod_Stri} is obtained by applying $\Omega (\xi ,\eta )=\xi^3+\eta^3$ in Corollary\ 3.4 in \cite{MP15}.

By using the same argument as in Lemma 2.3 in \cite{GTV97}, we obtain the following estimates
from Proposition~\ref{Stri} and Proposition~\ref{mod_Stri}. 
\begin{cor}
Let $(p,q)\in \R^2$ satisfy $p\ge 3$ and $\frac{3}{p}+\frac{2}{q}=1$. 
For $N$, $L\in 2^{\Z}$, we have
\begin{equation}\label{Stri_FR}
\|P_NQ_Lu\|_{L^p_tL^q_{xy}}\lesssim L^{\frac{1}{2}}\|P_NQ_Lu\|_{L^2_{txy}}.
\end{equation}
Furthermore, if $\F_{xy}[P_Nu]$ is supported in $\{(\xi,\eta)|\ |\xi |\sim |\eta |\}$, then we have
\begin{equation}\label{mStri_FR}
\|P_NQ_Lu\|_{L^4_{txy}}\lesssim N^{-\frac{1}{4}}L^{\frac{1}{2}}\|P_NQ_Lu\|_{L^2_{txy}}.
\end{equation}
\end{cor}
To get a positive power of $M$, we give the following estimates. 
\begin{cor}
Let $0<\delta \ll 1$, $0<\epsilon <1-\delta$. 
For $N$, $M$, $L\in 2^{\Z}$, we have
\begin{equation}\label{LP_Stri}
\|P_{N,M}Q_Lu\|_{L^{\frac{4}{1+\delta}}_{txy}}\lesssim 
(NM)^{\frac{\epsilon}{4}}L^{\frac{5(1-\delta)}{12}-\frac{\epsilon}{6}}\|P_{N,M}Q_Lu\|_{L^2_{txy}}.
\end{equation}
Furthermore, if 
$\F_{xy}[P_Nu]$ is supported in $\{(\xi,\eta)|\ |\xi |\sim |\eta |\}$, then we have
\begin{equation}\label{mLP_Stri}
\|P_{N,M}Q_Lu\|_{L^{\frac{4}{1+\delta}}_{txy}}\lesssim 
(NM)^{\frac{\epsilon}{4}}N^{-\frac{1}{4}(1-\delta-\epsilon)}
L^{\frac{1-\delta}{2}-\frac{\epsilon}{4}}\|P_NQ_Lu\|_{L^2_{txy}}.
\end{equation}
\end{cor}
%
%
\begin{proof}
By (\ref{Stri_FR}) with $p=q=5$ , we have the $L^5$-Strichartz estimate 
\begin{equation}\label{L5_Stri}
\|P_{N,M}Q_Lu\|_{L^{5}_{txy}}\lesssim L^{\frac{1}{2}}\|P_{N,M}Q_Lu\|_{L^2_{txy}}.
\end{equation}
By the interpolation between (\ref{L5_Stri})
and a trivial equality $\|P_{N,M}Q_Lu\|_{L^2_{txy}}=L^0\|P_{N,M}Q_Lu\|_{L^2_{txy}}$, 
we have 
\begin{equation}\label{L4-_Stri}
\|P_{N,M}Q_Lu\|_{L^{\frac{4-2\epsilon}{1+\delta}}_{txy}}\lesssim 
L^{\frac{5(1-\delta-\epsilon)}{6(2-\epsilon)}}\|P_NQ_Lu\|_{L^2_{txy}}.
\end{equation}
While, by the Cauchy-Schwartz inequality, we obtain
\[
\begin{split}
\|P_{N,M}Q_Lu\|_{L^{\infty}_{xy}}
&\le \int_{\substack{|(\xi,\eta)|\sim N\\ |\xi+\eta|\sim M}}|\F_{xy}[P_{N,M}Q_Lu](\xi,\eta )|d\xi d\eta\\
&\lesssim (NM)^{\frac{1}{2}}\|P_{N,M}Q_Lu\|_{L^{2}_{xy}}.
\end{split}
\]
Therefore, by using (\ref{Stri_FR}) with $(p,q)=(\infty ,2)$, we have
\begin{equation}\label{inf_Stri}
\|P_{N,M}Q_Lu\|_{L^{\infty}_{txy}}\lesssim (NML)^{\frac{1}{2}}\|P_{N,M}Q_Lu\|_{L^{2}_{txy}}.
\end{equation}
By the interpolation between (\ref{L4-_Stri}) and (\ref{inf_Stri}), 
we obtain (\ref{LP_Stri}). 

By using (\ref{mStri_FR}) instead of (\ref{L5_Stri}) in the above argument, we also get (\ref{mLP_Stri}).
\end{proof}
Next, we give the bilinear Strichartz  estimates. 
\begin{prop}\label{BSE}
Let $R_{K}^{(j)}$ ($j=1,2$) denote the bilinear operator defined by
\[
\begin{split}
\F_{xy}[R_K^{(1)}(u_1,u_2)](\xi, \eta )&=\int \varphi_{K}(\xi_1^2-(\xi -\xi_1)^2)
\widehat{u_1}(\xi_1,\eta_1)\widehat{u_2}(\xi -\xi_1,\eta -\eta_1)d\xi_1d\eta_1,\\
\F_{xy}[R_K^{(2)}(u_1,u_2)](\xi, \eta )&=\int \varphi_{K}(\eta_1^2-(\eta -\eta_1)^2)
\widehat{u_1}(\xi_1,\eta_1)\widehat{u_2}(\xi -\xi_1,\eta -\eta_1)d\xi_1d\eta_1.
\end{split}
\]
For $N_1$, $N_2$, $L_1$, $L_2$, $K\in 2^{\Z}$ with $N_1\ge N_2$, and $j\in \{1,2\}$, we have
\begin{equation}\label{BSE_1}
\begin{split}
&\|R_K^{(j)}(P_{N_1}Q_{L_1}u_1, P_{N_2}Q_{L_2}u_2)\|_{L^2_{txy}}\\
&\lesssim K^{-\frac{1}{2}}N_2^{\frac{1}{2}}L_1^{\frac{1}{2}}L_2^{\frac{1}{2}}\|P_{N_1}Q_{L_1}u_1\|_{L^2_{txy}}\|P_{N_2}Q_{L_2}u_2\|_{L^2_{txy}}.
\end{split}
\end{equation}
\end{prop}
\begin{proof}
We only prove for $j=1$ because the case $j=2$ can be proved by the same way. 
We put $f_i=\F[P_{N_i}Q_{L_i}u_i]$, $\zeta_i=(\xi_i,\eta_i)$ $(i=1,2)$. 
By the duality argument, it suffice to show that
\begin{equation}\label{BSE_pf_1}
\begin{split}
&\left|\int_{\Omega}f_1(\tau_1,\zeta_1)f_2(\tau_2,\zeta_2)f(\tau_1+\tau_2,\zeta_1+\zeta_2)d\tau_1d\tau_2d\zeta_1d\zeta_2\right|\\
&\lesssim K^{-\frac{1}{2}}N_2^{\frac{1}{2}}L_1^{\frac{1}{2}}L_2^{\frac{1}{2}}\|f_1\|_{L^2_{\tau \xi \eta}}\|f_2\|_{L^2_{\tau \xi \eta}}\|f\|_{L^2_{\tau \xi \eta}}
\end{split}
\end{equation}
for any $f\in L^2(\R\times \R^2)$, where
\[
\Omega =\{(\tau_1,\tau_2,\zeta_1,\zeta_2)|\ |\zeta_i|\sim N_i,\ |\tau_i-\xi_i^3-\eta_i^3|\sim L_i\ (i=1,2),\ |\xi_1^2-\xi_2^2|\sim K\}.
\]
By the Cauchy-Schwartz inequality, we have
\begin{equation}\label{BSE_pf_2}
\begin{split}
&\left|\int_{\Omega}f_1(\tau_1,\zeta_1)f_2(\tau_2,\zeta_2)f(\tau_1+\tau_2,\zeta_1+\zeta_2)d\tau_1d\tau_2d\zeta_1d\zeta_2\right|\\
&\lesssim \|f_1\|_{L^2_{\tau\xi\eta}}\|f_2\|_{L^2_{\tau\xi\eta}}
\left(\int_{\Omega}|f(\tau_1+\tau_2,\zeta_1+\zeta_2)|^2d\tau_1d\tau_2d\zeta_1d\zeta_2\right)^{\frac{1}{2}}.
\end{split}
\end{equation}
By applying the variable transform $(\tau_1,\tau_2)\mapsto (\theta_1,\theta_2)$ and $(\zeta_1,\zeta_2)\mapsto (\mu,w,z,\nu)$ as 
\[
\begin{split}
&\theta_i=\tau_i-\xi_i^3-\eta_i^3\ \  (i=1,2),\\
&\mu =\theta_1+\theta_2+\xi_1^3+\xi_2^3+\eta_1^3+\eta_2^3,\ w=\xi_1+\xi_2,\ z=\eta_1+\eta_2,\ \nu =\eta_2,
\end{split}
\]
we have
\[
\begin{split}
&\int_{\Omega}|f(\tau_1+\tau_2,\zeta_1+\zeta_2)|^2d\tau_1d\tau_2d\zeta_1d\zeta_2\\
&\lesssim \int_{\substack{|\theta_1|\sim L_1\\ |\theta_2|\sim L_2}}\left(\int_{|\nu |\lesssim N_2}|f(\mu,w,z)|^2
\ee_{\{|\xi_1^2-\xi_2^2|\sim K\}}(\xi_1,\xi_2)J(\zeta_1,\zeta_2)^{-1}d\mu dwdzd\nu \right)d\theta_1d\theta_2 ,
\end{split}
\]
where
\[
J(\zeta_1,\zeta_2)
=\left|{\rm det}\frac{\partial (\mu ,w,z,\nu )}{\partial (\xi_1,\eta_1,\xi_2,\eta_2)}\right|
=3|\xi_1^2-\xi_2^2|.
\]
Therefore, we obtain
\begin{equation}\label{BSE_pf_3}
\int_{\Omega}|f(\tau_1+\tau_2,\zeta_1+\zeta_2)|^2d\tau_1d\tau_2d\zeta_1d\zeta_2
\lesssim K^{-1}N_2L_1L_2\|f\|_{L^2_{\tau\xi\eta}}.
\end{equation}
As a result, we get (\ref{BSE_pf_1}) from (\ref{BSE_pf_2}) and (\ref{BSE_pf_3}).
\end{proof}
\begin{rem}
In particullar, 
if $N_1\gg N_2$, then we have
\begin{equation}\label{BSE_4}
\begin{split}
&\|P_{N_1}Q_{L_1}u_1\cdot P_{N_2}Q_{L_2}u_2\|_{L^{2}_{txy}}\\
&\lesssim N_1^{-1}N_2^{\frac{1}{2}}L_1^{\frac{1}{2}}L_2^{\frac{1}{2}}\|P_{N_1}Q_{L_1}u_1\|_{L^2_{txy}}\|P_{N_2}Q_{L_2}u_2\|_{L^2_{txy}}
\end{split}
\end{equation}
since the equality
\[
P_{N_1}Q_{L_1}u_1\cdot P_{N_2}Q_{L_2}u_2
=R_K^{(j)}(P_{N_1}Q_{L_1}u_1, P_{N_2}Q_{L_2}u_2)
\]
with $K\sim N_1^2$ holds for $j=1$ or $2$. 
\end{rem}
\begin{cor}
Let $0< \delta \ll 1$, $0<\epsilon <1-\delta$. 
For $N_1$, $N_2$, $M_1$, $M_2$, $L_1$, $L_2\in 2^{\Z}$ with 
$N_1\gg N_2$, we have
\begin{equation}\label{BSE_3}
\begin{split}
&\|P_{N_1,M_1}Q_{L_1}u_1\cdot P_{N_2,M_2}Q_{L_2}u_2\|_{L^{\frac{2}{1+\delta}}_{txy}}\\
&\lesssim J_{\delta,\epsilon}
(L_1L_2)^{\frac{1-\delta}{2}-\frac{\epsilon}{4}}\|P_{N_1,M_1}Q_{L_1}u_1\|_{L^2_{txy}}\|P_{N_2,M_2}Q_{L_2}u_2\|_{L^2_{txy}},
\end{split}
\end{equation}
where
\[
J_{\delta,\epsilon}=J_{\delta,\epsilon}(N_1,M_1,N_2,M_2)=
(N_1M_1N_2M_2)^{\frac{\epsilon}{4}}
(N_1^{-1}N_2^\frac{1}{2})^{1-\delta-\epsilon}.
\]
\end{cor}
%
%
\begin{proof}
By the H\"older inequality and (\ref{inf_Stri}), we have 
\[
\begin{split}
&\|P_{N_1,M_1}Q_{L_1}u_1\cdot P_{N_2,M_2}Q_{L_2}u_2\|_{L^{2}_{txy}}\\
&\lesssim
\|P_{N_1,M_1}Q_{L_1}u_1\|_{L^{\infty}_{txy}}^{\frac{1}{2}}
\|P_{N_2,M_2}Q_{L_2}u_2\|_{L^{2}_{txy}}^{\frac{1}{2}}
\|P_{N_1,M_1}Q_{L_1}u_1\|_{L^{2}_{txy}}^{\frac{1}{2}}
\|P_{N_2,M_2}Q_{L_2}u_2\|_{L^{\infty}_{txy}}^{\frac{1}{2}}\\
&\lesssim (N_1M_1L_1N_2M_2L_2)^{\frac{1}{4}}
\|P_{N_1,M_1}Q_{L_1}u_1\|_{L^{2}_{txy}}\|P_{N_2,M_2}Q_{L_2}u_2\|_{L^{2}_{txy}}.
\end{split}
\]
By the interpolation between this estimate and (\ref{BSE_4}), we obtain
\begin{equation}\label{BSE_delta}
\begin{split}
&\|P_{N_1,M_1}Q_{L_1}u_1\cdot P_{N_2,M_2}Q_{L_2}u_2\|_{L^{2}_{txy}}\\
&\lesssim J_{\delta,\epsilon}^{\frac{1}{1-\delta}}
(L_1L_2)^{\frac{1}{2}-\frac{\epsilon}{4(1-\delta )}}\|P_{N_1,M_1}Q_{L_1}u_1\|_{L^2_{txy}}\|P_{N_2,M_2}Q_{L_2}u_2\|_{L^2_{txy}}.
\end{split}
\end{equation}
While, by the Cauchy-Schwartz inequality, we have
\begin{equation}\label{L1_bilin}
\|P_{N_1,M_1}Q_{L_1}u_1\cdot P_{N_2,M_2}Q_{L_2}u_2\|_{L^1_{txy}}
\lesssim \|P_{N_1,M_1}Q_{L_1}u_1\|_{L^2_{txy}}\|P_{N_2,M_2}Q_{L_2}u_2\|_{L^2_{txy}}.
\end{equation}
By the interpolation between (\ref{BSE_delta}) and (\ref{L1_bilin}), 
we obtain (\ref{BSE_3}). 
\end{proof}
Here, we prove Proposition~\ref{bilin_est}.
\begin{proof}[Proof of Proposition~\ref{bilin_est}]
By using the embedding $l^1\hookrightarrow l^2$ for the summation $\sum_N\sum_M$,  and the duality argument, we have
\[
\begin{split}
&\|(\partial_x+\partial_y)(uv)\|_{X^{s,-\frac{1}{2},1}}\\
&\lesssim \sum_{N_1,M_1,L_1}\sum_{N_2,M_2,L_2}\left(\sum_{N,M,L}
\frac{\langle N\rangle^sM}{\langle M^2+L\rangle^{\frac{1}{2}}}\right.\\
&\hspace{14ex}\times \left. \sup_{\|w\|_{L^2}=1}\left|\int P_{N_1,M_1}Q_{L_1}u\cdot P_{N_2,M_2}Q_{L_2}v\cdot P_{N,M}Q_Lwdtdxdy\right|\right). 
\end{split}
\]
We put 
\[
u_{N_1,M_1,L_1}=P_{N_1,M_1}Q_{L_1}u,\ v_{N_2,M_2,L_2}=P_{N_2,M_2}Q_{L_2}v,\ w_{N,M,L}=P_{N,M}Q_Lw,
\]
\[
f_{N_1,M_1,L_1}=\langle N_1\rangle^s\langle M_1^2+L_1\rangle^{\frac{1-\delta}{2}}u_{N_1,M_1,L_1},\ 
g_{N_2,M_2,L_2}=\langle N_2\rangle^s\langle M_2^2+L_2\rangle^{\frac{1-\delta}{2}}v_{N_2,M_2,L_2},
\]
for $0<\delta \ll 1$ and
\[
I=\left|\int u_{N_1,M_1,L_1}\cdot v_{N_2,M_2,L_2}\cdot w_{N,M,L}dtdxdy\right|.
\]
We note that $L_1^{b}\|u_{N_1,M_1,L_1}\|_{L^2_{txy}}
\lesssim \langle N_1\rangle^{-s}\|f_{N_1,M_1,L_1}\|_{L^2_{txy}}$ 
and $L_2^{b}\|v_{N_2,M_2,L_2}\|_{L^2_{txy}}
\lesssim \langle N_2\rangle^{-s}\|g_{N_2,M_2,L_2}\|_{L^2_{txy}}$
hold for $b\le \frac{1-\delta}{2}$ 
since $L_i\lesssim \langle M_i^2+L_i\rangle$ $(i=1,2)$. 

By the symmetry, we can assume $N_1\gtrsim N_2$. 
We first consider the case $1\ge N_1\gtrsim N_2$. 
We note that
\begin{equation}\label{L2p_Stri}
\|P_{N,M}Q_Lu\|_{L^{\frac{2}{1-\delta}}_{txy}}
\lesssim L^{\frac{5}{6}\delta}\|P_{N,M}Q_Lu\|_{L^2_{txy}}
\end{equation}
holds by the interpolation between (\ref{L5_Stri}) 
and a trivial equality $\|P_{N,M}Q_Lu\|_{L^2_{txy}}=L^0\|P_{N,M}Q_Lu\|_{L^2_{txy}}$.
By the H\"older inequality,  (\ref{LP_Stri}), and (\ref{L2p_Stri}), we have
\[
\begin{split}
I&\lesssim \|u_{N_1,M_1,L_1}\|_{L^{\frac{4}{1+\delta}}_{txy}}
\|v_{N_2,M_2,L_2}\|_{L^{\frac{4}{1+\delta}}_{txy}}
\|w_{N,M,L}\|_{L^{\frac{2}{1-\delta}}_{txy}}\\
&\lesssim 
(N_1M_1N_2M_2)^{\frac{\epsilon}{2}} L^{\frac{5}{6}\delta}\|f_{N_1,M_1,L_1}\|_{L^2_{txy}}\|g_{N_2,M_2,L_2}\|_{L^2_{txy}}\|w_{N,M,L}\|_{L^2_{txy}}
\end{split}
\]
since $\langle N_i\rangle^{s}\sim 1$ $(i=1,2)$ for any $s\in \R$. 
Therefore, we obtain
\begin{equation}\label{low_freq_ineq}
\begin{split}
&\sum_{N\lesssim 1}\sum_{M\lesssim N}\sum_{L}
\frac{\langle N\rangle^sM}{\langle M^2+L\rangle^{\frac{1}{2}}}\sup_{\|w\|_{L^2}=1}I\\
&\lesssim 
(N_1M_1N_2M_2)^{\frac{\epsilon}{2}}\|f_{N_1,M_1,L_1}\|_{L^2_{txy}}\|g_{N_2,M_2,L_2}\|_{L^2_{txy}}
\end{split}
\end{equation}
since
\[
\sum_{L}\frac{L^{\frac{5}{6}\delta}}{\langle M^2+L\rangle^{\frac{1}{2}}}
\lesssim \sum_{L\lesssim \langle M\rangle^2}\frac{L^{\frac{5}{6}\delta}}{\langle M\rangle}
+\sum_{L\gtrsim \langle M\rangle^2}L^{-(\frac{1}{2}-\frac{5}{6}\delta)}
\lesssim \langle M\rangle^{-(1-\frac{5}{3}\delta)}\lesssim 1
\]
and
\[
\sum_{N\lesssim 1}\sum_{M\lesssim N}
\langle N\rangle^sM
\sim \sum_{N\lesssim 1}\sum_{M\lesssim N}M
\lesssim \sum_{N\lesssim 1}N\lesssim 1
\]
for any $s\in \R$. 
By using  (\ref{low_freq_ineq})and the Cauchy-Schwartz inequality for the summations
$\sum_{N_1,M_1\lesssim 1}$ and $\sum_{N_2,M_2\lesssim 1}$, we have
\[
\begin{split}
&\sum_{N_1,M_1\lesssim 1}\sum_{L_1}\sum_{N_2,M_2\lesssim 1}\sum_{L_2}
\left(\sum_{N,M,L}\frac{\langle N\rangle^sM}{\langle M^2+L\rangle^{\frac{1}{2}}}\sup_{\|w\|_{L^2}=1}I\right)
\lesssim 
\|u\|_{X^{s,\frac{1-\delta}{2},1}}\|v\|_{X^{s,\frac{1-\delta}{2},1}}
\end{split}
\]
for any $s\in \R$. 

Next, we consider the case $N_1\gtrsim N_2$, $N_1\ge 1$. 
It suffice to show that
\begin{equation}\label{bilin_pf}
\begin{split}
&\sum_{N,M,L}\frac{\langle N\rangle^sM}{\langle M^2+L\rangle^{\frac{1}{2}}}
\sup_{\|w\|_{L^2}=1}I\\
&\lesssim 
N_{1}^{-\epsilon}(M_1M_2)^{\frac{\epsilon}{4}}
\|f_{N_1,M_1,L_1}\|_{L^2_{txy}}\|g_{N_2,M_2,L_2}\|_{L^2_{txy}}
\end{split}
\end{equation}
for small $\epsilon >0$. Indeed, 
(\ref{bilin_pf}) and the Cauchy-Schwartz inequality for the summations 
$\sum_{N_1,M_1}$ and $\sum_{N_2,M_2}$ imply
\[
\begin{split}
&\sum_{\substack{N_1,M_1,L_1\\ N_1\ge 1}}\sum_{\substack{N_2,M_2,L_2\\N_2\lesssim N_1}}
\left(\sum_{N,M,L}\frac{\langle N\rangle^sM}{\langle M^2+L\rangle^{\frac{1}{2}}}\sup_{\|w\|_{L^2}=1}I\right)\\
&\lesssim 
\left(\sum_{N_1\ge 1}\sum_{M_1\lesssim N_1}
\sum_{N_2\lesssim N_1}\sum_{M_2\lesssim N_2}
N_1^{-2\epsilon}(M_1M_2)^{\frac{\epsilon}{2}}\right)^{\frac{1}{2}}\\
&\ \ \ \ \times 
\left\{\sum_{N_1}\sum_{M_1}\left(\sum_{L_1}\|f_{N_1,M_1,L_1}\|_{L^2_{txy}}\right)^2\right\}^{\frac{1}{2}}
\left\{\sum_{N_2}\sum_{M_2}\left(\sum_{L_2}\|g_{N_2,M_2,L_2}\|_{L^2_{txy}}\right)^2\right\}^{\frac{1}{2}}\\
&\lesssim \|u\|_{X^{s,\frac{1-\delta}{2},1}}\|v\|_{X^{s,\frac{1-\delta}{2},1}}.
\end{split}
\]
Now, we prove (\ref{bilin_pf}). \\
\kuuhaku \\
\underline{Case\ 1:\ $N_1\sim N_2\gg N$,\ $N_1\ge 1$.}

We note that $M\lesssim \max\{M_1,M_2\}$ since $\xi +\eta =(\xi_1+\eta_1)+(\xi-\xi_1+\eta -\eta_1)$. 
By the symmetry, we can assume $M\lesssim M_1$. 
By the H\"older inequality, we have
\[
I\lesssim \|u_{N_1,M_1,L_1}\|_{L^{\frac{2}{1-\delta}}_{txy}}
\|v_{N_2,M_2,L_2}\cdot w_{N,M,L}\|_{L^{\frac{2}{1+\delta}}_{txy}}. 
\]
Furthermore, we have
\[
\|u_{N_1,M_1,L_1}\|_{L^{\frac{2}{1-\delta}}_{txy}}
\lesssim L_1^{\frac{5}{6}}\|u_{N_1,M_1,L_1}\|_{L^2_{txy}}
\lesssim \frac{N_1^{-s}}{\langle M_1^2+L_1\rangle^{\frac{1-\delta}{2}-\frac{5}{6}\delta}}\|f_{N_1,M_1,L_1}\|_{L^2_{txy}}
\]
by (\ref{L2p_Stri}), and we have
\[
\begin{split}
&\|v_{N_2,M_2,L_2}\cdot w_{N,M,L}\|_{L^{\frac{2}{1+\delta}}_{txy}}\\
&\lesssim J_{\delta,\epsilon}(N_2,M_2,N,M)(L_2L)^{\frac{1-\delta}{2}-\frac{\epsilon}{4}}
\|v_{N_2,M_2,L_2}\|_{L^2_{txy}}\|w_{N,M,L}\|_{L^2_{txy}}\\
&\lesssim  (M_1M_2)^{\frac{\epsilon}{4}}N_2^{-s-1+\delta+\frac{5}{4}\epsilon}
N^{\frac{1-\delta}{2}-\frac{\epsilon}{4}}L^{\frac{1-\delta}{2}-\frac{\epsilon}{4}}\|g_{N_2,M_2,L_2}\|_{L^2_{txy}}\|w_{N,M,L}\|_{L^2_{txy}}
\end{split}
\]
by (\ref{BSE_3}) and $M\lesssim M_1$. 
Therefore, if we choose $\epsilon >0$ as $\epsilon =\frac{10}{3}\delta$, we obtain
\[
\begin{split}
&\sum_{N\ll N_1}\sum_{M\lesssim M_1}\sum_{L}
\frac{\langle N\rangle^sM}{\langle M^2+L\rangle^{\frac{1}{2}}}\sup_{\|w\|_{L^2}=1}I\\
&\lesssim (M_1M_2)^{\frac{\epsilon}{4}}N_1^{-s}N_2^{-s-1+\delta+\frac{5}{4}\epsilon}
\|f_{N_1,M_1,L_1}\|_{L^2_{txy}}\|g_{N_2,M_2,L_2}\|_{L^2_{txy}}\\
&\hspace{10ex}\times \left(\sum_{N\ll N_1}\langle N\rangle^sN^{\frac{1-\delta}{2}-\frac{\epsilon}{4}}\sum_{M\lesssim M_1}
\frac{M}{\langle M_1^2+L_1\rangle^{\frac{1-\delta}{2}-\frac{5}{6}\delta}}
\sum_{L}\frac{L^{\frac{1-\delta}{2}-\frac{\epsilon}{4}}}{\langle M^2+L\rangle^{\frac{1}{2}}}\right)\\
&\lesssim N_1^{-\epsilon}(M_1M_2)^{\frac{\epsilon}{4}}N_1^{-s-\frac{1}{2}+\frac{\delta}{2}+2\epsilon}\|f_{N_1,M_1,L_1}\|_{L^2_{txy}}\|g_{N_2,M_2,L_2}\|_{L^2_{txy}}
\end{split}
\]
for $s\ge -\frac{1-\delta}{2}+\frac{\epsilon}{4}$ since
\[
\sum_{M\lesssim M_1}
\frac{M}{\langle M_1^2+L_1\rangle^{\frac{1-\delta}{2}-\frac{5}{6}\delta}}
\sum_{L}\frac{L^{\frac{1-\delta}{2}-\frac{\epsilon}{4}}}{\langle M^2+L\rangle^{\frac{1}{2}}}
\lesssim \sum_{M\lesssim M_1}
\frac{M^{1-\delta-\frac{\epsilon}{2}}}{\langle M_1\rangle^{1-\frac{8}{3}\delta}}
\lesssim 1. 
\]
As a result, we get (\ref{bilin_pf}) for $s> -\frac{1}{2}$ 
if we choose $\delta >0$ as $0<\delta <\frac{6}{43}\left(s+\frac{1}{2}\right)$. \\
\\
\underline{Case\ 2:\ $N\sim N_1\gg N_2$,\ $N_1\ge 1$. }

By the H\"older inequality, we have
\[
I\lesssim \|u_{N_1,M_1,L_1}\cdot v_{N_2,M_2,L_2}\|_{L^\frac{2}{1+\delta}_{txy}}
\|w_{N,M,L}\|_{L^\frac{2}{1-\delta}_{txy}}. 
\]
Furthermore, we have
\[
\|w_{N,M,L}\|_{L^\frac{2}{1-\delta}_{txy}}
\lesssim L^{\frac{5}{6}\delta}\|w_{N,M,L}\|_{L^2_{txy}}
\]
by (\ref{L2p_Stri}), and we have
\[
\begin{split}
&\|u_{N_1,M_1,L_1}\cdot v_{N_2,M_2,L_2}\|_{L^\frac{2}{1+\delta}_{txy}}\\
&\lesssim J_{\delta,\epsilon}(N_1,M_1,N_2,M_2)(L_1L_2)^{\frac{1-\delta}{2}-\frac{\epsilon}{4}}
\|u_{N_1,M_1,L_1}\|_{L^2_{txy}}\|v_{N_2,M_2,L_2}\|_{L^2_{txy}}\\
&\lesssim (M_1M_2)^{\frac{\epsilon}{4}}N_1^{-s-1+\delta+\frac{5}{4}\epsilon}
\langle N_2\rangle^{-s}N_2^{\frac{1-\delta}{2}-\frac{\epsilon}{4}}
\|f_{N_1,M_1,L_1}\|_{L^2_{txy}}\|g_{N_2,M_2,L_2}\|_{L^2_{txy}}
\end{split}
\]
by (\ref{BSE_3}). 
Therefore, if $s\le \frac{1-\delta}{2}-\frac{\epsilon}{4}$, we obtain
\[
\begin{split}
&\sum_{N\sim N_1}\sum_{M\lesssim N}\sum_{L}
\frac{\langle N\rangle^sM}{\langle M^2+L\rangle^{\frac{1}{2}}}\sup_{\|w\|_{L^2}=1}I\\
&\lesssim (M_1M_2)^{\frac{\epsilon}{4}}N_1^{-s-\frac{1-\delta}{2}+\epsilon}
\|f_{N_1,M_1,L_1}\|_{L^2_{txy}}\|g_{N_2,M_2,L_2}\|_{L^2_{txy}}
\left(\sum_{M\lesssim N_1}M
\sum_{L}\frac{L^{\frac{5}{6}\delta}}{\langle M^2+L\rangle^{\frac{1}{2}}}\right)\\
&\lesssim N_1^{-\epsilon}(M_1M_2)^{\frac{\epsilon}{4}}
N_1^{-s-\frac{1}{2}+\frac{13}{6}\delta+2\epsilon}\|f_{N_1,M_1,L_1}\|_{L^2_{txy}}\|g_{N_2,M_2,L_2}\|_{L^2_{txy}},
\end{split}
\]
since
\[
\sum_{L}\frac{L^{\frac{5}{6}\delta}}{\langle M^2+L\rangle^{\frac{1}{2}}}
\lesssim M^{-(1-\frac{5}{3}\delta)}. 
\]
As a result, we get (\ref{bilin_pf}) for $\frac{1}{2}>s> -\frac{1}{2}$ 
if we choose $\delta >0$ and $\epsilon >0$ as
$0<\epsilon <\frac{1}{2}(s+\frac{1}{2})$, $0<\delta <\min\left\{\frac{6}{13}\left(s+\frac{1}{2}-2\epsilon\right), 2\left(\frac{1}{2}-s-\frac{\epsilon}{2}\right)\right\}$. 

While if $s\ge \frac{1}{2}$, then we have
\[
I\lesssim (M_1M_2)^{\frac{\epsilon}{4}}N_1^{-s-\frac{1-\delta}{2}+\epsilon}L^{\frac{5}{6}\delta}
\|f_{N_1,M_1,L_1}\|_{L^2_{txy}}\|g_{N_2,M_2,L_2}\|_{L^2_{txy}}\|w_{N,M,L}\|_{L^2_{txy}}
\]
by the same argument with using $\langle N_2\rangle^{-s}\lesssim 1$. 
Therefore, we obtain
\[
\begin{split}
&\sum_{N\sim N_1}\sum_{M\lesssim N}\sum_{L}
\frac{\langle N\rangle^sM}{\langle M^2+L\rangle^{\frac{1}{2}}}\sup_{\|w\|_{L^2}=1}I\\
&\lesssim N_1^{-\epsilon}(M_1M_2)^{\frac{\epsilon}{4}}
N_1^{-\frac{1}{2}+\frac{13}{6}\delta+2\epsilon}\|f_{N_1,M_1,L_1}\|_{L^2_{txy}}\|g_{N_2,M_2,L_2}\|_{L^2_{txy}},
\end{split}
\]
which implies (\ref{bilin_pf}) since $-\frac{1}{2}+\frac{13}{6}\delta+2\epsilon<0$.\\
\\
\underline{Case\ 3:\ $N\sim N_1\sim N_2\ge 1$}

We can assume $M\lesssim M_1$ such as Case\ 1.
We split $v_{N_2,M_2,L_2}$ and $w_{N,M,L}$ into
\[
v_{N_2,M_2,L_2}=\sum_{i=1}^3R_{i}v_{N_2,M_2,L_2},\ \ w_{N,M,L}=\sum_{j=1}^3R_{j}w_{N,M,L}. 
\]
We put
\[
I_{i,j}=\left|\int u_{N_1,M_1,L_1}\cdot R_{i}v_{N_2,M_2,L_2}\cdot R_{j}w_{N,M,L}dtdxdy\right|, 
\]
where $R_i$ $(i=1,2,3)$ are projections given by
\[
\F_{xy}[R_1f]=\ee_{\{|\xi |\gg |\eta |\}}\widehat{f},\ \F_{xy}[R_2f]=\ee_{\{|\xi |\sim |\eta |\}}\widehat{f},\ \F_{xy}[R_3f]=\ee_{\{|\xi |\ll |\eta |\}}\widehat{f}.
\]
We note that $\F_{xy}[w_{N,M,L}]$ is supported in at least one of $\{(\xi,\eta)|\ |\xi|\sim N\}$ or $\{(\xi,\eta)|\ |\eta |\sim N\}$. 
By the symmetry, we can assume $\supp \F_{xy}[w_{N,M,L}]\subset \{(\xi,\eta)|\ |\xi|\sim N\}$. 
Then, it suffice to show the estimate for $I_{i,j}$ with $i=1,2,3$, $j=1,2$. \\
\\
\underline{Estimate for $I_{1,1}$}

In this case, we note that $N\sim N_1\sim N_2\sim M\sim M_1\sim M_2$ and
\[
|\xi \xi_1\xi_2+\eta \eta_1\eta_2|\sim |\xi \xi_1 \xi_2|\sim N_1^3
\]
for $(\xi_1,\eta_1)\in \supp \F_{xy}[u_{N_1,M_1,L_1}]$, 
$(\xi_2,\eta_2)\in \supp \F_{xy}[v_{N_2,M_2,L_2}]$
with $\xi_1+\xi_2=\xi$, $\eta_1+\eta_2=\eta$. 
It implies
\[
\max\{L_1,L_2,L\}\gtrsim N_1^3
\]
since 
\[
|(\tau_1-\xi_1^3-\eta_1^3)+(\tau_2-\xi_2^3-\eta_2^3)-(\tau-\xi^3-\eta^3)|
=3|\xi\xi_1\xi_2+\eta\eta_1\eta_2|.
\]
holds for $(\tau_i,\xi_i,\eta_i)$ $(i=1,2)$ with $(\tau,\xi,\eta)=(\tau_1+\tau_2,\xi_1+\xi_2,\eta_1+\eta_2)$.\\
\kuuhaku \\
(i)\ For the case $L\gtrsim N_1^3$

By the H\"older inequality, (\ref{LP_Stri}), and (\ref{L2p_Stri}), we have
\[
\begin{split}
I&\lesssim \|u_{N_1,M_1,L_1}\|_{L^{\frac{4}{1+\delta}}_{txy}}
\|v_{N_2,M_2,L_2}\|_{L^{\frac{4}{1+\delta}}_{txy}}
\|w_{N,M,L}\|_{L^{\frac{2}{1-\delta}}_{txy}}\\
&\lesssim 
(N_1M_1N_2M_2)^{\frac{\epsilon}{4}}
(L_1L_2)^{\frac{5(1-\delta)}{12}-\frac{\epsilon}{6}}L^{\frac{5}{6}\delta}\|u_{N_1,M_1,L_1}\|_{L^2_{txy}}\|v_{N_2,M_2,L_2}\|_{L^2_{txy}}\|w_{N,M,L}\|_{L^2_{txy}}\\
&\sim N_1^{-\epsilon}(M_1M_2)^{\frac{\epsilon}{4}}N_1^{-2s+\frac{3}{2}\epsilon}L^{\frac{5}{6}\delta}\|f_{N_1,M_1,L_1}\|_{L^2_{txy}}\|g_{N_2,M_2,L_2}\|_{L^2_{txy}}\|w_{N,M,L}\|_{L^2_{txy}}.
\end{split}
\]
Therefore, we obtain
\[
\begin{split}
&\sum_{N\sim N_1}\sum_{M\lesssim N}\sum_{L\gtrsim N_1^3}
\frac{\langle N\rangle^sM}{\langle M^2+L\rangle^{\frac{1}{2}}}\sup_{\|w\|_{L^2}=1}I\\
&\lesssim N_1^{-\epsilon}(M_1M_2)^{\frac{\epsilon}{4}}
N_1^{-s+\frac{3}{2}\epsilon}
\|f_{N_1,M_1,L_1}\|_{L^2_{txy}}\|g_{N_2,M_2,L_2}\|_{L^2_{txy}}
\left(\sum_{M\lesssim N_1}M
\sum_{L\gtrsim N_1^3}\frac{L^{\frac{5}{6}\delta}}{\langle M^2+L\rangle^{\frac{1}{2}}}\right)\\
&\lesssim N_1^{-\epsilon}(M_1M_2)^{\frac{\epsilon}{4}}N_1^{-s-\frac{1}{2}+\frac{5}{2}\delta+\frac{3}{2}\epsilon}\|f_{N_1,M_1,L_1}\|_{L^2_{txy}}\|g_{N_2,M_2,L_2}\|_{L^2_{txy}},
\end{split}
\]
since
\[
\sum_{L\gtrsim N_1^3}\frac{L^{\frac{5}{6}\delta}}{\langle M^2+L\rangle^{\frac{1}{2}}}
\lesssim \sum_{L\gtrsim N_1^3}L^{-(\frac{1}{2}-\frac{5}{6}\delta)}
\lesssim N_1^{-\frac{3}{2}+\frac{5}{2}\delta}. 
\]
As a result, we get (\ref{bilin_pf}) for $s> -\frac{1}{2}$ 
if we choose $\delta >0$ and $\epsilon >0$ as 
$0<\epsilon <\frac{2}{3}(s+\frac{1}{2})$, $0<\delta <\frac{2}{5}\left(s+\frac{1}{2}-\frac{3}{2}\epsilon\right)$. \\
\\
(ii)\ For the case $L_1\gtrsim N_1^3$

By the H\"older inequality, (\ref{L2p_Stri}), and (\ref{LP_Stri}), we have
\[
\begin{split}
I&\lesssim \|u_{N_1,M_1,L_1}\|_{L^{\frac{2}{1-\delta}}_{txy}}
\|v_{N_2,M_2,L_2}\|_{L^{\frac{4}{1+\delta}}_{txy}}
\|w_{N,M,L}\|_{L^{\frac{4}{1+\delta}}_{txy}}\\
&\lesssim 
L_1^{\frac{5}{6}\delta}(N_2M_2NM)^{\frac{\epsilon}{4}}
(L_2L)^{\frac{5(1-\delta)}{12}-\frac{\epsilon}{6}}\|u_{N_1,M_1,L_1}\|_{L^2_{txy}}\|v_{N_2,M_2,L_2}\|_{L^2_{txy}}\|w_{N,M,L}\|_{L^2_{txy}}\\
&\lesssim N_1^{-\epsilon}(M_1M_2)^{\frac{\epsilon}{4}}
N_1^{-2s-\frac{3}{2}+4\delta+\frac{3}{2}\epsilon}
L^{\frac{5(1-\delta)}{12}-\frac{\epsilon}{6}}\|f_{N_1,M_1,L_1}\|_{L^2_{txy}}\|g_{N_2,M_2,L_2}\|_{L^2_{txy}}\|w_{N,M,L}\|_{L^2_{txy}}
\end{split}
\]
since 
$L_1^{\frac{5}{6}\delta}\langle M_1^2+L_1\rangle^{-\frac{1-\delta }{2}}
\lesssim L_1^{-(\frac{1}{2}-\frac{4}{3}\delta)}\lesssim N_1^{-\frac{3}{2}+4\delta}$. 
Therefore, we obtain
\[
\begin{split}
&\sum_{N\sim N_1}\sum_{M\lesssim N}\sum_{L}
\frac{\langle N\rangle^sM}{\langle M^2+L\rangle^{\frac{1}{2}}}\sup_{\|w\|_{L^2}=1}I\\
&\lesssim N_1^{-\epsilon}(M_1M_2)^{\frac{\epsilon}{4}}
N_1^{-s-\frac{3}{2}+4\delta+\frac{3}{2}\epsilon}
\|f_{N_1,M_1,L_1}\|_{L^2_{txy}}\|g_{N_2,M_2,L_2}\|_{L^2_{txy}}
\left(\sum_{M\lesssim N_1}M
\sum_{L}\frac{L^{\frac{5(1-\delta )-2\epsilon}{12}}}{\langle M^2+L\rangle^{\frac{1}{2}}}\right)\\
&\lesssim N_1^{-\epsilon}(M_1M_2)^{\frac{\epsilon}{4}}
N_1^{-s-\frac{2}{3}+\frac{19}{6}\delta+\frac{7}{6}\epsilon}\|f_{N_1,M_1,L_1}\|_{L^2_{txy}}\|g_{N_2,M_2,L_2}\|_{L^2_{txy}},
\end{split}
\]
since
\[
\sum_{L}\frac{L^{\frac{5(1-\delta )-2\epsilon}{12}}}{\langle M^2+L\rangle^{\frac{1}{2}}}
\lesssim 
M^{-\frac{1+5\delta+2\epsilon}{6}}. 
\]
As a result, we get (\ref{bilin_pf}) for $s> -\frac{2}{3}$ 
if we choose $\delta >0$ and $\epsilon >0$ as
$0<\epsilon <\frac{6}{7}(s+\frac{1}{2})$, $0<\delta <\frac{6}{19}\left(s+\frac{2}{3}-\frac{7}{6}\epsilon\right)$. 
The  case $L_2\gtrsim N^3$ is same.\\
\\
\underline{Estimate for $I_{2,2}$}

In this case, we have
\[
|\xi_2|\sim |\eta_2|\sim N_2,\ |\xi|\sim |\eta|\sim N.
\]
By the H\"older inequality, (\ref{L2p_Stri}), (\ref{mLP_Stri}), 
and $M\lesssim M_1$, we have
\[
\begin{split}
I&\lesssim \|u_{N_1,M_1,L_1}\|_{L^{\frac{2}{1-\delta}}_{txy}}
\|v_{N_2,M_2,L_2}\|_{L^{\frac{4}{1+\delta}}_{txy}}
\|w_{N,M,L}\|_{L^{\frac{4}{1+\delta}}_{txy}}\\
&\lesssim 
L_1^{\frac{5}{6}\delta}(N_2M_2NM)^{\frac{\epsilon}{4}}(N_2N)^{-\frac{1-\delta-\epsilon}{4}}
(L_2L)^{\frac{1-\delta}{2}-\frac{\epsilon}{4}}\|u_{N_1,M_1,L_1}\|_{L^2_{txy}}\|v_{N_2,M_2,L_2}\|_{L^2_{txy}}\|w_{N,M,L}\|_{L^2_{txy}}\\
&\lesssim (M_1M_2)^{\frac{\epsilon}{4}}
N_1^{-2s-\frac{1-\delta}{2}+\epsilon}\frac{L^{\frac{1-\delta}{2}-\frac{\epsilon}{4}}}{\langle M_1^2+L_1\rangle^{\frac{1-\delta}{2}-\frac{5}{6}\delta}}\|f_{N_1,M_1,L_1}\|_{L^2_{txy}}\|g_{N_2,M_2,L_2}\|_{L^2_{txy}}\|w_{N,M,L}\|_{L^2_{txy}}.
\end{split}
\]
Therefore,  we get (\ref{bilin_pf}) for $s> -\frac{1}{2}$ by the same argument as in Case 1. \\
\\
\underline{Estimate for $I_{1,2}$}

In this case, we have
\[
|\eta_2^2-\eta^2|\sim |\eta|^2 \sim N^2
\]
Therefore, we obtain
\[
\begin{split}
&\|R_{1}v_{N_2,M_2,L_2}\cdot R_{2}w_{N,M,L}\|_{L^{2}_{txy}}
\lesssim N^{-\frac{1}{2}}L_2^{\frac{1}{2}}L^{\frac{1}{2}}\|R_{1}v_{N_2,M_2,L_2}\|_{L^2_{txy}}\|R_{2}w_{N,M,L}\|_{L^2_{txy}}
\end{split}
\]
by (\ref{BSE_1}) since 
\[
R_{1}v_{N_2,M_2,L_2}\cdot R_{2}w_{N,M,L}
=R_{K}^{(2)}(R_{1}v_{N_2,M_2,L_2}\cdot R_{2}w_{N,M,L})
\]
with $K\sim N$ holds. 
While, by the Cauchy-Schwartz inequality, we have
\[
\|R_{1}u_{N_1,M_1,L_1}\cdot R_{j}v_{N_2,M_2,L_2}\|_{L^{1}_{txy}}
\lesssim \|R_{1}v_{N_2,M_2,L_2}\|_{L^2_{txy}}\|R_{2}w_{N,M,L}\|_{L^2_{txy}}.
\]
Therefore, we obtain the bilinear Stirchartz estimate such as (\ref{BSE_3}) for the
product $R_{1}v_{N_2,M_2,L_2}\cdot R_{2}w_{N,M,L}$ ,
and we get (\ref{bilin_pf}) for $s>-\frac{1}{2}$ 
by the same argument as in Case 1 since $M\lesssim M_1$.
The estimates for $I_{2,1}$, $I_{3,1}$, and $I_{3,2}$ are 
obtained by the same way. 
\end{proof}
\begin{rem}\label{be_mod_rem}
We can also obtain the bilinear estimate
\[
\|(\partial_x+\partial_y)(uv)\|_{X^{s,-\frac{1}{2},1}}
\le \frac{C_3}{2}\left(\|u\|_{X^{s,\frac{1-\delta}{2},1}}\|v\|_{X^{s_0,\frac{1-\delta}{2},1}}
+\|u\|_{X^{s_0,\frac{1-\delta}{2},1}}\|v\|_{X^{s,\frac{1-\delta}{2},1}}\right)
\]
for $s\ge s_0>-\frac{1}{2}$ by using 
\[
\langle \xi\rangle^s\lesssim \langle \xi\rangle^{s_0}
\left(\langle \xi_1\rangle^{s-s_0}+\langle \xi-\xi_1\rangle^{s-s_0}\right). 
\]
\end{rem}
%
%
\section{Proof of the well-posedness}
In this section, we prove Theorem~\ref{LWP} and \ref{GWP}. 
For $T>0$ and $v_0\in H^s(\R^2)$, we define the map $\Phi_{T, v_0}$ as
\[
\Phi_{T, v_0}(v)(t):=\psi(t)\left(
W(t)u_0 +\int_0^tW(t-t')(\partial_x+\partial_y)(\psi_T(t')^2v(t')^2)dt'
\right), 
\] 
where $\psi$ is cut-off function defined in Section\ 2, 
and $\psi_T(t)=\psi\left(\frac{t}{T}\right)$.
For $R>0$ and Banach space $X$, we define 
$B_R(X):=\{u\in X|\ \|u\|_{X}\le R\}$. 
To obtain the well-posedness of (\ref{ZKB_sym}) in $H^s(\R^2)$, 
we prove that $\Phi_{T, v_0}$ is a contraction map 
on closed subset of $X^{s,\frac{1}{2},1}$. 
\begin{lemm}\label{sol_loc}
Let $0<T\le 1$, $0<\delta \le 1$.  
There exist $C_4>0$ and $\mu =\mu (\delta )>0$, such that for any $u\in X^{s,\frac{1}{2},1}$, 
we have
\[
\|\psi_Tu\|_{X^{s,\frac{1-\delta}{2},1}}
\le C_4T^{\mu}\|u\|_{X^{s,\frac{1}{2},1}}. 
\]
\end{lemm}
The proof of Lemma~\ref{sol_loc} is almost same 
as the proof of Lemma~2.5 and 3.1 in \cite{GTV97}. 
\begin{proof}[Proof of Theorem~\ref{LWP}]
Let $s\ge s_0>-\frac{1}{2}$ and $v_{0}\in H^s(\R^2)$ are given, 
and $T\in (0,1]$, $R>0$ will be chosen later. 
We define the function space $Z^s$ as
\[
Z^{s}:=\{v\in X^{s,\frac{1}{2},1}|\ \|v\|_{Z^s}:=\|v\|_{X^{s_0,\frac{1}{2},1}}+\alpha \|v\|_{X^{s,\frac{1}{2},1}}<\infty\}, 
\]
where $\alpha=\|v_0\|_{H^{s_0}}/\|v_0\|_{H^s}$. 
For $v$, $v_1$, $v_2\in B_R(Z^{s})$, we have
\[
\begin{split}
\|\Phi_{T,v_{0}}(v)\|_{Z^{s}}
&\le C_1(1+\alpha)\|v_{0}\|_{H^{s_0}} +C_2C_3C_4^2T^{2\mu}\|v\|_{Z^{s}}^2\\
&\leq C_1(1+\alpha )\|v_0\|_{H^{s_0}}+C_2C_3C_4^2T^{2\mu}R^2
\end{split}
\]
and
\[
\begin{split}
\|\Phi_{T,v_{0}}(v_1)-\Phi_{T,v_{0}}(v_2)\|_{Z^{s}}
&\leq C_2C_3C_4^2T^{2\mu}\|v_1+v_2\|_{Z^{s}}
\|v_1-v_2\|_{Z^{s}}\\
&\leq C_2C_3C_4^2T^{2\mu}R\|v_1-v_2\|_{Z^{s}}
\end{split}
\]
by Proposition~\ref{lin_est}, ~\ref{duam_est}, ~\ref{bilin_est}, Remark~\ref{be_mod_rem}, 
and Lemma~\ref{sol_loc}. 
Therefore, if we choose $T$, $R$ as
\[
R=2C_1(1+\alpha )\|v_0\|_{H^{s_0}},\ 0<T^{2\mu}<(4C_1C_2C_3C_4^2(1+\alpha)\|v_0\|_{H^{s_0}})^{-1}, 
\]
then $\Phi_{T,v_0}$ is contraction map on $B_R(Z^{s})$. 
We note that $T=T(\|v_0\|_{H^{s_0}})$. 
By Banach's fixed point theorem, there exists 
a solution $v\in X^{s,\frac{1}{2},1}$ to $v(t)=\Phi_{T,v_0}(v)(t)$ 
and $v|_{[0,T]}\in X^{s,\frac{1}{2},1}_T$ 
satisfies (\ref{ZKB_sym_int}) on $[0,T]$. 
The Lipschitz continuous dependence on initial data is obtained 
by the similar argument as above. 
The uniqueness is 
obtained by the same argument as in Section\ 4.2 of \cite{MR02}.  
\end{proof}

Next, to prove the global well-posedness of (\ref{ZKB_sym}) in $\widetilde{H}^{s}(\R^2)$, 
we define the function space $\widetilde{X}^{s,b,1}$ as the completion of the Schwartz class ${\mathcal S}(\R_{t}\times \R^2_{x,y})$ with the norm
\[
\|u\|_{\widetilde{X}^{s,b,1}}=\left\{\sum_{N\in 2^{\Z}}\sum_{M\in 2^{\Z}}\left(\sum_{L\in 2^{\Z}}\langle M\rangle^s\langle M^2+L\rangle^{b}\|P_{N,M}Q_{L}u\|_{L^2_{txy}}\right)^2\right\}^{\frac{1}{2}}.
\]
We also define $\widetilde{X}^{s,b,1}_T$ as the time localized space of $\widetilde{X}^{s,b,1}$. 
\begin{rem}
We can see that $\widetilde{X}^{s,\frac{1}{2},1}_T\hookrightarrow L^2((0,T);\widetilde{H}^{s+1}(\R^2))$ since $\langle M\rangle^{s+1} \lesssim \langle M\rangle^{s}\langle M^2+L\rangle^{\frac{1}{2}}$ and $l^1_L\hookrightarrow l^2_L$ hold. 
\end{rem}
\begin{prop}\label{lin_est_gwp}
Let $s\in \R$. There exists $C_1>0$, such that for any $u_0\in \widetilde{H}^s(\R^2)$, we have
\[
\|\psi (t)W(t)u_0\|_{\widetilde{X}^{s,\frac{1}{2},1}}\le C_1\|u_0\|_{\widetilde{H}^s}. 
\]
\end{prop}
\begin{prop}\label{duam_est_gwp}
Let $s\in \R$. There exists $C_2>0$, such that for any $F\in \widetilde{X}^{s,-\frac{1}{2},1}$, we have
\[
\left\|\psi (t){\mathcal L}F(t)\right\|_{\widetilde{X}^{s,\frac{1}{2},1}}\le C_2\|F\|_{\widetilde{X}^{s,-\frac{1}{2},1}}
\]
\end{prop}
The proof of Proposition~\ref{lin_est_gwp} and ~\ref{duam_est_gwp} 
are same as the proof of Proposition~\ref{lin_est} and ~\ref{duam_est}. 
\begin{prop}\label{bilin_est_gwp}
Let $s>-\frac{1}{2}$. There exist $0<\delta \ll 1$ and $C_3>0$, such that for any $u$, $v\in \widetilde{X}^{s,\frac{1-\delta}{2},1}$, we have
\[
\|(\partial_x+\partial_y)(uv)\|_{\widetilde{X}^{s,-\frac{1}{2},1}}\le C_3\|u\|_{\widetilde{X}^{s,\frac{1-\delta}{2},1}}\|v\|_{\widetilde{X}^{s,\frac{1-\delta}{2},1}}.
\]
\end{prop}
The proof of Proposition~\ref{bilin_est_gwp} 
is similar to the proof of Proposition~\ref{bilin_est}. 
We will give the proof at the last part of this section. 
\begin{proof}[Proof of Theorem~\ref{GWP}]
Let $s\ge s_0>-\frac{1}{2}$ are given. 
By Proposition~\ref{lin_est_gwp},~\ref{duam_est_gwp},~\ref{bilin_est_gwp}, 
and using the same argument as in the proof of Theorem~\ref{LWP}, 
we obtain the solution $v\in \widetilde{X}^{s,\frac{1}{2},1}_T$ to 
(\ref{ZKB_sym}) on $[0,T]$ with $T=T(\|v_0\|_{\widetilde{H}^{s_0}})$. 
Let $T'\in (0,T)$ be fixed. 
Since $\widetilde{X}^{s,\frac{1}{2},1}_T\hookrightarrow L^2([0,T];\widetilde{H}^{s+1}(\R^2))$ 
holds, there exists $t_0\in (0,T')$ 
such that $v(t_0)\in \widetilde{H}^{s+1}(\R^2)$. 
Therefore, by choosing $v(t_0)$ as the initial data 
and using the uniqueness of the solution, we obtain 
$v(t_0+\cdot)\in \widetilde{X}^{s+1,\frac{1}{2},1}_{T-t_0}$. 
In particular, we have $v(T')\in \widetilde{H}^{s+1}(\R^2)$. 
By repeating this argument, we get $v(T')\in \widetilde{H}^{\infty}(\R^2)$. 
Since we can choose $T'>0$ arbitrary small, $v$ belongs to 
$C((0,T];\widetilde{H}^{\infty}(\R^2))$. 
This arrows us to take the $L^2$-scalar product of (\ref{ZKB_sym}) with $v$, 
 and we have
\[
\begin{split}
\frac{d}{dt}\|v(t)\|_{L^2_x}^2
&=(\partial_tv(t), v(t))_{L^2_x}\\
&=\left(-(\partial_x^3+\partial_y^3)v(t)+(\partial_x+\partial_y)^2v(t)+(\partial_x+\partial_y)(v(t)^2),v(t)\right)_{L^2_x}\\
&=-\|(\partial_x+\partial_y)v(t)\|_{L^2_x}^2\le 0
\end{split}
\]
for any $t\in (0,T)$. 
Therefore, $\|v(t)\|_{L^2_x}$ is non-increasing, 
and we can extend the solution $v$ globally in time. 
\end{proof}
\begin{rem}
We note that the embedding
$X^{s,\frac{1}{2},1}_T\hookrightarrow L^2([0,T];H^{s+1}(\R^2))$ 
does not hold. 
Therefore, we cannot use the above argument 
for initial data $v_0\in H^s(\R^2)$. 
\end{rem}
Finally, we give the proof of Proposition~\ref{bilin_est_gwp}
\begin{proof}[Proof of Proposition~\ref{bilin_est_gwp}]
We put 
\[
u_{N_1,M_1,L_1}=P_{N_1,M_1}Q_{L_1}u,\ v_{N_2,M_2,L_2}=P_{N_2,M_2}Q_{L_2}v,\ w_{N,M,L}=P_{N,M}Q_Lw,
\]
\[
f_{N_1,M_1,L_1}=\langle M_1\rangle^s\langle M_1^2+L_1\rangle^{\frac{1-\delta}{2}}u_{N_1,M_1,L_1},\ 
g_{N_2,M_2,L_2}=\langle M_2\rangle^s\langle M_2^2+L_2\rangle^{\frac{1-\delta}{2}}v_{N_2,M_2,L_2}
\]
for $0<\delta \ll 1$ and
\[
I=\left|\int u_{N_1,M_1,L_1}\cdot v_{N_2,M_2,L_2}\cdot w_{N,M,L}dtdxdy\right|.
\]
We use $L_1^{b}\|u_{N_1,M_1,L_1}\|_{L^2_{txy}}
\lesssim \langle M_1\rangle^{-s}\|f_{N_1,M_1,L_1}\|_{L^2_{txy}}$ 
and $L_2^{b}\|v_{N_2,M_2,L_2}\|_{L^2_{txy}}
\lesssim \langle M_2\rangle^{-s}\|g_{N_2,M_2,L_2}\|_{L^2_{txy}}$ 
instead of 
$L_1^{b}\|u_{N_1,M_1,L_1}\|_{L^2_{txy}}
\lesssim \langle N_1\rangle^{-s}\|f_{N_1,M_1,L_1}\|_{L^2_{txy}}$ 
and $L_2^{b}\|v_{N_2,M_2,L_2}\|_{L^2_{txy}}
\lesssim \langle N_2\rangle^{-s}\|g_{N_2,M_2,L_2}\|_{L^2_{txy}}$
in the proof of Proposition~\ref{bilin_est}. 
By the same argument as in the proof of Proposition~\ref{bilin_est}, 
we have
\[
\begin{split}
&\sum_{N_1,M_1\lesssim 1}\sum_{L_1}\sum_{N_2,M_2\lesssim 1}\sum_{L_2}
\left(\sum_{N,M,L}\frac{\langle M\rangle^sM}{\langle M^2+L\rangle^{\frac{1}{2}}}\sup_{\|w\|_{L^2}=1}I\right)
\lesssim 
\|u\|_{\widetilde{X}^{s,\frac{1-\delta}{2},1}}\|v\|_{\widetilde{X}^{s,\frac{1-\delta}{2},1}}
\end{split}
\]
for any $s\in \R$ and it suffice to show that
\begin{equation}\label{bilin_pf_gwp}
\begin{split}
&\sum_{N,M,L}\frac{\langle M\rangle^sM}{\langle M^2+L\rangle^{\frac{1}{2}}}
\sup_{\|w\|_{L^2}=1}I\\
&\lesssim 
N_{1}^{-\epsilon}(M_1M_2)^{\frac{\epsilon}{4}}
\|f_{N_1,M_1,L_1}\|_{L^2_{txy}}\|g_{N_2,M_2,L_2}\|_{L^2_{txy}}
\end{split}
\end{equation}
for $N_1\ge N_2$, $N_1\ge 1$, and small $\epsilon >0$. \\
\kuuhaku \\
\kuuhaku \\
\underline{Case\ 1':\ $N_1\sim N_2\gg N$}

We only have to modify little in the proof of Proposition~\ref{bilin_est}, Case\ 1.
Since it hold that
\[
\langle M_1\rangle^{-s}
\sum_{M\lesssim M_1}
\frac{\langle M\rangle^sM}{\langle M_1^2+L_1\rangle^{\frac{1-\delta}{2}-\frac{5}{6}\delta}}
\sum_{L}\frac{L^{\frac{1-\delta}{2}-\frac{\epsilon}{4}}}{\langle M^2+L\rangle^{\frac{1}{2}}}
\lesssim \sum_{M\lesssim M_1}\frac{M^{s+1-\delta-\frac{\epsilon}{2}}}{\langle M_1\rangle^{s+1-\frac{8}{3}\delta}}
\lesssim 1
\]
for $\epsilon =\frac{10}{3}\delta$, $s>-1+\frac{8}{3}\delta$, and 
\[
\langle M_2\rangle^{-s}\lesssim N_1^{-s}
\]
for $s<0$, 
we get (\ref{bilin_pf}) for $-\frac{1}{2}<s<0$ 
by the same way as in the proof of Proposition~\ref{bilin_est}, Case\ 1. \\
\\
\underline{Case\ 2':\ $N\sim N_1\gg N_2$}

If $M\ge M_1$, then we have
\[
\langle M\rangle^{s}\langle M_1\rangle^{-s}\langle M_2\rangle^{-s}
\lesssim \langle M_2\rangle^{-s}\lesssim N_1^{-s}
\]
for $s<0$. 
Therefore, we get (\ref{bilin_pf}) for $-\frac{1}{2}<s<0$ 
by the same way as in the proof of Proposition~\ref{bilin_est}, Case\ 2.

While, if $M\le M_1$, then we have
\[
J_{\delta,\epsilon}(N,M,N_2,M_2)
\lesssim J_{\delta,\epsilon}(N_1,M_1,N_2,M_2). 
\]
Therefore, by estimating
\[
I\lesssim \|u_{N_1,M_1,L_1}\|_{L^\frac{2}{1-\delta}_{txy}}
\|v_{N_2,M_2,L_2}\cdot w_{N,M,L}\|_{L^\frac{2}{1+\delta}_{txy}}
\]
instead of
\[
I\lesssim 
\|u_{N_1,M_1,L_1}\cdot v_{N_2,M_2,L_2}\|_{L^\frac{2}{1+\delta}_{txy}}
\|w_{N,M,L}\|_{L^\frac{2}{1-\delta}_{txy}}
\]
in the proof of Proposition~\ref{bilin_est}, Case\ 2, 
we get (\ref{bilin_pf}) for $-\frac{1}{2}<s<0$ 
by the same modification such as Case\ 1'\\
\kuuhaku \\
\underline{Case\ 3':\ $N\sim N_1\sim N_2\ge 1$}

If ${\rm supp}\F_{x,y}[w_{N,M,L}]\subset \{(\xi, \eta)|\ |\xi|\gg |\eta|\ {\rm or}\ 
|\xi|\ll |\eta|\}$, then  $M\sim N$ holds. 
Therefore, we have
\[
\langle M\rangle^{s}\langle M_1\rangle^{-s}\langle M_2\rangle^{-s}
\lesssim \langle N\rangle^{s}\langle N_1\rangle^{-s}\langle N_2\rangle^{-s}\lesssim N_1^{-s}
\]
for $s<0$ and get (\ref{bilin_pf}) for $-\frac{1}{2}<s<0$ 
by the same way as in the proof of Proposition~\ref{bilin_est}, Case\ 3. 

We assume ${\rm supp}\F_{x,y}[w_{N,M,L}]\subset \{(\xi, \eta)|\ |\xi|\sim |\eta|\}$. 
It suffice to show the estimate for $I_{1,2}$ and $I_{2,2}$, 
which are defined in Proposition~\ref{bilin_est}, Case\ 3. 
By the same modification such as in Case\ 1', we can obtain (\ref{bilin_pf}) for $-\frac{1}{2}<s<0$. 
\end{proof}
\section*{Acknowledgements}
This work is financially supported by JSPS KAKENHI Grant Number 17K14220
and Program to Disseminate Tenure Tracking System from the Ministry of Education, Culture, Sports, Science and Technology. 
The author would like to his appreciation 
to Shinya Kinoshita (Nagoya university) 
for his useful comments and discussions. 

\end{document}